\documentclass[12pt]{article}
      \usepackage{setspace}     
     \usepackage[pdftex]{graphicx}
     \usepackage{sidecap}
\usepackage{fancyhdr, amsmath, amssymb, amsfonts, url, dsfont, mathrsfs, graphicx, epsfig,  amsthm,mathrsfs}
\usepackage{tikz}
\usepackage{wrapfig}
\usepackage{framed}
\usepackage{listings}

\usepackage{pgf}
\usepackage{tikz}
\usetikzlibrary{arrows,automata}

\usepackage{dsfont}
\usepackage{float}
\usepackage{verbatim} 
\usepackage[latin1]{inputenc}
\usepackage{cite}
\usepackage{graphicx}
\usepackage{vmargin} 
\usepackage{ifpdf} 
\usepackage{url} 
\usepackage{fancyhdr}
\usepackage{lmodern}
\usepackage{moreverb}
\usepackage{enumitem}
\usepackage{amsmath}
\usepackage{pgf}
\usepackage{tikz}
\usepackage{graphicx}
\usetikzlibrary{arrows}
\usetikzlibrary{shapes,snakes,calendar,matrix,backgrounds,folding}

\usepackage{color}
\definecolor{rltred}{rgb}{0.75,0,0}
\definecolor{rltgreen}{rgb}{0,0.5,0}
\definecolor{rltblue}{rgb}{0,0,0.75}

\usepackage{thumbpdf}
\usepackage[pdftex,colorlinks=true,urlcolor=rltred, filecolor=rltgreen, linkcolor=rltblue]{hyperref}

      \usepackage{caption}
      \usepackage{subcaption}
      \usepackage{breqn}	
      \theoremstyle{plain}
      \newtheorem{theorem}{Theorem}[section]
      \newtheorem{lemma}[theorem]{Lemma}
      \newtheorem{corollary}[theorem]{Corollary}
      \newtheorem{remark}[theorem]{Remark}
      \newtheorem{prop}[theorem]{Proposition}

      \newtheorem{definition}[theorem]{Definition}
      \newtheorem{example}[theorem]{Example}

      \newcommand{\R}{{\mathbb R}}

      \newcommand{\B}{\mathcal{B}}
      
      \newcommand{\N}{\mathbb{N}}

      \newcommand{\C}{\mathbb{C}}

      \newcommand{\1}{\mathds{1}}   
      
      \newcommand{\vertiii}[1]{{\left\vert\kern-0.25ex\left\vert\kern-0.25ex\left\vert #1 
    \right\vert\kern-0.25ex\right\vert\kern-0.25ex\right\vert}}
    
\begin{document}

\title{The smoothness of the stationary measure}

\author{Italo Cipriano}
\date{April 2016}

\maketitle

\begin{abstract} We study the smoothness of the stationary measure with respect to smooth perturbations of the iterated function scheme and the weight functions that define it. Our main theorems relate the smoothness of the perturbation of:  the iterated function scheme and the weight functions; to the smoothness of the perturbation of the stationary measure. The results depend on the smoothness of:  the iterated function scheme and the weights functions; and the space on which the stationary measure acts as a linear operator. As a consequence we also obtain the smoothness of the Hausdorff dimension of the limit set and of the Hausdorff dimension of the stationary measure.\end{abstract}
  
\section{Introduction}\label{sec:intro}

An IFS (iterated function scheme) with constant weight functions has a unique stationary measure associated sometimes also called self-similar measure. Self-similar measures were originally defined in \cite{Hutchinson}. The most studied features of IFSs (iterated function schemes) are their fractal properties, like the Hausdorff dimension of its limit set \cite{Hutchinson,Falconer} and the Hausdorff dimension of its stationary probability measure \cite{Manning1981,Ledrappier81,Young_82_Dim,Geronimo89,GDMS2003}. In this paper we are concerned with analytic properties of conformal IFSs, mostly motivated by \cite{Ruelle_HD_83,Mane90,Pollicott2015}. A particularly natural special case is that of a finite family of contractions on the unit interval. For definiteness, let us consider the following setting:

\begin{definition}\label{DEFnumOne}
Assume that $\epsilon>0$ small, $\beta,\varepsilon>0,$ $k,l,m\in\N\setminus\{1\},$  $r\in\N$ and call the interval $(-\epsilon,\epsilon)\subset \R$ by $\mathcal I_\epsilon.$ Then
\begin{enumerate}
\item let $\mathcal T^{(\lambda)}=\{T_i^{(\lambda)}\}_{i=1}^k$ with $\lambda\in \mathcal I_\epsilon$ be a family of $\mathcal{C}^{m+\beta}$ contractions on $[0,1].$ Assume that we can expand for $\lambda\in \mathcal I_\epsilon$, 
$$T_{i}^{(\lambda)}=T_{i}+\lambda T_{i,1}+\cdots+ \lambda^{m-1} T_{i,m-1} + o(\lambda^{m-1}),$$
where $T_i,T_{i,j}\in \mathcal{C}^{m+\beta}([0,1],[0,1]),$ $\|dT_i\|_{\mathcal{C}^1}<1,$ $dT_i=dT_j$ for $i\in \{ 1 ,\ldots, k\} $ and $j\in \{ 1 ,\ldots, m-1\};$ and 
\item let $\mathcal G^{(\theta)}=\{g_i^{(\theta)}\}_{i=1}^k$ with $\theta\in \mathcal I_\epsilon$ be a family  of $\mathcal{C}^{l+\varepsilon}([0,1],\mathbb R^+)$ positive weight functions on $[0,1]$ satisfying the following two conditions: 
\begin{equation}\label{normalization_cond}
\sum_{i=1}^kg_i^{(\theta)}\equiv 1 \mbox{ and }
\end{equation}
\begin{equation}\label{my_cond}
\sum_{i=1}^k \left\| g_i^{(\theta)}\right\|_{\mathcal C^0}  Lip\left(T^{(\lambda)}_i \right)<1 \mbox{ for all }\lambda,\theta\in \mathcal I_{\epsilon}
\end{equation}
where
$$g_{i}^{(\theta)}=g_{i}+\theta g_{i,1}+\cdots+ \theta^{r} g_{i,r} + o(\theta^{r}) \mbox{ and}$$
 $$g_i,g_{i,j}\in \mathcal{C}^{l+\varepsilon}([0,1],\mathbb R^+)\mbox{ for }i \in \{ 1 ,\ldots, k\} \mbox{ and }j \in \{ 1 ,\ldots, r \} .$$
\end{enumerate}
\end{definition}

In this case the stationary measure  $\mu = \mu_{\lambda, \theta}$ is the unique probability measure on $[0,1]$ that satisfies 
\begin{equation}\label{24_07_2015_time:1:52}
\int f(x) d\mu (x) = \sum_{i=1}^k \int g^{(\theta)}_i(x) f(T^{(\lambda)}_ix)d\mu(x)
\end{equation}
 for any continuous function $f:[0,1] \to \mathbb R$.

The existence of such a measure is well known and discussed in Subsection \ref{sec:22_07_2015_stationarymeasures}. There is an equivalent definition of stationary measure which is perhaps somewhat more intuitive and particularly useful for simulations that is given by the following rather well known lemma.

\begin{lemma}\label{known_lemma_3sep}
For any $x_0 \in [0,1]$ we can write $\mu$ as the weak star limit of finitely supported probability measures, indeed
$$
\mu = \lim_{n\to+\infty} \sum_{\underline i  
\in \{1,\cdots, k\}^n } g_{\underline i}^{(\theta)}(x_0) \delta_{T^{(\lambda)}_{\underline i }(x_0)}, 
$$
where for each of the $k^n$ strings  $\underline i = (i_1, \cdots, i_n) $  we write (for $n \in\N$):
$$
\begin{aligned}
T^{(\lambda)}_{\underline i} &:= T^{(\lambda)}_{i_1} \circ  \cdots \circ T^{(\lambda)}_{ i_n}: \mathbb R \to \mathbb R;\\
g_{\underline i}^{(\theta)} (x_0) &:=
g_{i_1}^{(\theta)}\left(T^{(\lambda)}_{i_{2}} \cdots T^{(\lambda)}_{i_n}(x_0)\right)\cdots g_{i_{n-1}}^{(\theta)}\left(T^{(\lambda)}_{i_n}(x_0)\right)\cdot g_{i_n}^{(\theta)} (x_0);\mbox{ and }\\
\delta_{T^{(\lambda)}_{\underline i }(x_0)}& \mbox{ denotes the Dirac measure supported on }T^{(\lambda)}_{\underline i }(x_0).
\end{aligned}
$$
 \end{lemma}

Our first main result is about the differentiability of the dependence of this measure.

\begin{theorem}\label{main_theorem}
Assume  $\delta\in (0,1),$ $k,l,m,s\in\N\setminus\{1\}$ and $r\in\N,$ then:  
\begin{enumerate}
\item Given $\theta \in \mathcal I_\epsilon,$ the measure $\mu_{\lambda,\theta}$ has a 
$\mathcal{C}^{\min (l,m,s)-1}$ dependence on $\lambda\in \mathcal I_\epsilon$ as an element of $\mathcal{C}^{s+\delta}([0,1],\mathbb R)^*.$
\item \label{sec:main_theorem_consequence}Given $\lambda \in \mathcal I_\epsilon,$ the measure $\mu_{\lambda,\theta}$ has a $\mathcal{C}^{r}$ dependence on $\theta \in \mathcal I_\epsilon$ as an element of $\mathcal{C}^{1}([0,1],\mathbb R)^*.$
\end{enumerate} 
\end{theorem}

\begin{remark}$\mbox{}$
 In Theorem \ref{main_theorem}, when we study the dependence of the measure $\mu=\mu_{\lambda,\theta}$  
on  $\lambda,$ it is essential to consider the measure $\mu$ as an element of $\mathcal{C}^{s+\delta}([0,1],\mathbb R)^*$ for $s\in \N\setminus\{1\}$, i.e. we identify $\mu$ with the functional $\mathscr{M}:\mathcal{C}^{s+\delta}([0,1],\mathbb R)\to \mathbb R$ defined by $\mathcal{C}^{s+\delta}([0,1],\mathbb R) \ni w\mapsto \int_0^1 w(\tilde{x}) d\mu(\tilde{x})\in \mathbb R.$ 
\end{remark}

We have the following simple corollary from Theorem \ref{main_theorem}.

\begin{corollary}
Let  $w: [0,1]  \to \mathbb R$
be a   $\mathcal{C}^\infty$ function.
 Given $\theta \in \mathcal I_\epsilon,$ the function 
 $(-\epsilon, \epsilon) \ni \lambda \mapsto \int w d\mu_{\lambda,\theta} \in \mathbb R$ is 
$\mathcal{C}^{\min (l,m)-1}$.
\end{corollary}

The next corollary applies under the  hypothesis that the  weight functions are $\mathcal{C}^{\infty}$. In particular, this is true in the special case of constant weight  functions.

\begin{corollary}
Suppose   that the family $\mathcal G^{(\theta)}=\{g_i^{(\theta)}\}_{i=1}^k$ of weights satisfies $g_i^{(\theta)}\in \mathcal{C}^{\infty}([0,1],\mathbb R^+)$ for every $i\in \{ 1 ,\ldots, k\} $.
Let  $w: [0,1]  \to \mathbb R$
be a   $\mathcal{C}^\infty$ function.
Given $\theta \in \mathcal I_\epsilon,$ the function 
 $(-\epsilon, \epsilon) \ni \lambda \mapsto \int w d\mu_{\lambda,\theta} \in \mathbb R$ is 
$\mathcal{C}^{m-1}$.
\end{corollary}

Our second result is on the differentiability of the Hausdorff dimension of the limit set $\mathcal{K}_{\lambda}$ of $\mathcal{T}^{(\lambda)}.$

\begin{theorem}\label{second_main_theorem}
Let $\mathcal T$ be an IFS as in Definition \ref{DEFnumOne} such that the sets $T_i^{(\lambda)}[0,1]$ are pairwise  disjoint for $i\in \{ 1 ,\ldots, k\}.$
Then  the dependence 
$(-\epsilon, \epsilon)\ni \lambda \mapsto HD( \mathcal{K}_{\lambda})$ 
of the Hausdorff dimension of the limit set of $\mathcal{T}^{(\lambda)},$ 
is $\mathcal{C}^{m-2}$.
\end{theorem}

Our last result is on the differentiability of the Hausdorff dimension of the stationary measure.

\begin{theorem}\label{third_main_theorem}
Let $\delta\in (0,1),$ $k,l,m,s\in\N\setminus\{1\}$ and $r\in\N.$ Consider $\mathcal T^{(\lambda)}=\{T_i^{(\lambda)}\}_{i=1}^k$ and $\mathcal G^{(\theta)}=\{g_i^{(\theta)}\}_{i=1}^k$ be as in Definition \ref{DEFnumOne} with the property that the sets $T_i^{(\lambda)}[0,1]$ are pairwise  disjoint for $i\in \{ 1 ,\ldots, k\}.$ If there exists $\rho>0$ such that
\begin{equation}\label{cond_17_Mar_2016}
\min_{i}\inf_{\lambda}\inf_{x}|dT_i^{(\lambda)}(x)|>\rho,
\end{equation}  
then 
\begin{enumerate}
\item given $\theta \in \mathcal I_\epsilon,$ 
the dependence $(-\epsilon, \epsilon)\ni \lambda \mapsto HD( \mu_{\lambda,\theta})$
of the Hausdorff dimension of the measure $\mu_{\lambda,\theta},$ is
$\mathcal{C}^{\min (l-1,m-2)};$ and  
\item given $\lambda \in \mathcal I_\epsilon,$ 
the dependence $(-\epsilon, \epsilon)\ni \theta \mapsto HD( \mu_{\lambda,\theta})$
of the Hausdorff dimension of the measure $\mu_{\lambda,\theta},$ is
 $\mathcal{C}^{r}.$
\end{enumerate} 
\end{theorem}

Our results use basic facts of IFS and are closely related to \cite{Tanaka}, see Subsection \ref{ExTanaka}. However, our proof relies on a result of composition of operators in \cite{DLLLave} and structural stability, whereas the proof in \cite{Tanaka} uses Proposition 2.3 in \cite{Tanaka} and \cite{Tanaka_2011}. \\

The structure of the paper is the following: In Section \ref{Sec_Background} we explain the background, in particular, we define and justify the existence and unicity of the stationary measures, we define the Hausdorff dimension of a limit set and the Hausdorff dimension of a stationary measure, we also state Bowen's formula and the volume lemma.
In Section \ref{PES_prel} we prove our main results. Finally, in Section \ref{sec_examples} we exhibit some examples of application of our results.\\

We are grateful to Mark Pollicott for suggesting most of the results and many of the ideas used in their proofs. We are also grateful to Ian Melbourne and Thomas Jordan for many useful remarks, corrections to the original notes and the suggestion of stating a result about the Hausdorff dimension of the stationary measure.

\section{Background}\label{Sec_Background}

We introduce iterated functions schemes, limit sets, stationary measures, projection maps, some basic results on thermodynamic formalism and the Hausdorff dimension of sets and measures.

\subsection{Stationary measures}\label{sec:22_07_2015_stationarymeasures}

We are only concerned with the study of stationary measures for IFSs, i.e. for a finite family of contractions with respect to the Lipchitz norm on a complete metric space. To make this precise, consider two complete metric spaces $(\mathcal M,d)$ and $(\mathcal N,\tilde{d}).$ Define the Lipschitz semi norm $Lip$ of $A:\mathcal M\to \mathcal N$ by 
$$
Lip(A):=\sup_{x\neq y} \frac{\tilde{d}(A(x),A(y))}{d(x,y)}.
$$

\begin{definition}[Iteration function scheme]
An IFS is a finite family of contractions with respect to $Lip$, i.e. a family of maps $\mathcal T=\{T_i\}_{i=1}^n$ where  $T_i:\mathcal M\to \mathcal M$ and $$\max_{i=1,\ldots,n}Lip(T_i)<1.$$
\end{definition}

Given a finite family of contractions, an interesting class of sets to study are those invariant under the contractions. The next lemma says that in the case of IFSs there exists a unique such set.

\begin{lemma}\label{limit_set_definition}
If $\mathcal T=\{T_i\}_{i=1}^n$ be an IFS, then there exists a unique closed bounded set $\mathcal K\subset \mathcal M$ such that 
$$
\mathcal K=\cup_{i=1}^nT_i\mathcal K.
$$
We call $\mathcal K$ the limit set of $\mathcal T.$
\end{lemma}

  A proof of this can be found in \cite{Hutchinson}. A basic example to keep in mind is the case of $(\mathcal M,d)=([0,1],|\mbox{ }|),$ for the unit interval $[0,1]$ and the absolute value $|\mbox{ }|$ on $\R,$ and $T_1(x)=\frac{x}{3},$ $T_2(x)=\frac{x}{3}+\frac{2}{3}.$ The limit set in this example is the famous middle third Cantor set.\\ 

An important property of an IFS is the open set condition that was introduced in \cite{MoranOSC} to compute the Hausdorff dimension of the limit set.

\begin{definition}[Open set condition]
An IFS $\mathcal T=\{T_i\}_{i=1}^n$ is said to satisfy the open set condition if there is a non-empty open set $\mathcal V\subset \mathcal M$ such that
\begin{equation}\label{OSC}
\cup_{i=1}^n T_i(\mathcal V)\subset \mathcal V \mbox{ and }T_i(\mathcal V)\cap T_j(\mathcal V)=\emptyset \mbox{ for }i\neq j.
\end{equation}
We say that $\mathcal T$ satisfies the open set condition for the open set being $\mathcal V,$ if $\mathcal V\subset \mathcal M$ such that (\ref{OSC}).
\end{definition}

Associated to an IFS $\mathcal T=\{T_i\}_{i=1}^n,$ we can consider a family of weight functions $\mathcal G=\{g_i\}_{i=1}^n,$ $g_i:\mathcal M\to (0,1)$ such that 
\begin{equation}\label{con1_21_07_2015}
\sum_{i=1}^n g_i\equiv 1 \mbox{ and}
\end{equation}
\begin{equation}\label{con2_21_07_2015}
\sum_{i=1}^n  \|g_i\|Lip(T_i)< 1,
\end{equation}
where $\|g\|=\sup\{g(x):x\in\mathcal M\}.$

\begin{definition}[Stationary measure]
Let $\mathcal T=\{T_i\}_{i=1}^n$ be an IFS with weight functions $\mathcal G=\{g_i\}_{i=1}^n$ and let $\mathcal{P}(\mathcal M)$ be the set of Borel regular probability measures having bounded support. A stationary measure $\mu\in \mathcal{P}(\mathcal M)$ is a fixed point for the operator $\mathscr{S}=\mathscr{S}_{\mathcal T,\mathcal G}:\mathcal{P}(\mathcal M)\to \mathcal{P}(\mathcal M)$ defined by  
$$
\mathscr{S}(\nu)(f):=\sum_{i=1}^n\int g_i(x)f(T_i(x))d\nu(x),
$$
where $\nu\in \mathcal{P}(\mathcal M)$ and $f : \mathcal M \to \R$ is a continuous compactly supported function. \end{definition}
\begin{remark} Two direct but important facts from the definition of stationary measure are the following:
\begin{enumerate}
\item A stationary measure for $(\mathcal T, \mathcal G)$ is supported on the limit set of $\mathcal T$ (a proof is given in \cite{Hutchinson}, Section 4.4).
\item A probability measure $\mu\in \mathcal{P}(\mathcal M)$ is a fixed point of $\mathscr{S}$ if and only if 
$$
\mathscr{S}(\mu)(f)=\int f(x)d\mu(x)
$$
for every continuous compactly supported function $f : \mathcal M \to \R.$
\end{enumerate}
\end{remark}
We have the following well known theorem:
\begin{theorem}
Suppose that $\mathcal M$ is a compact metric space.  An IFS $\mathcal T$ with weight functions $\mathcal G$ satisfying (\ref{con1_21_07_2015}) and (\ref{con2_21_07_2015}) has a unique stationary measure.
\end{theorem}

A proof of this theorem can be found in \cite{Hutchinson} for constant weight functions, using the contractive mapping principle. A small modification of the same argument can be applied here. Recall also that the existence of a stationary measure is a classic result \cite{furstenber1963} (Lemma 1.2).

\begin{proof}
The space $\mathcal P(\mathcal M)$  can be equipped with the Kantorovich-Rubinshtein norm \cite{Bogachev}
$$
\vertiii{\mu}=\sup\left\{\int f d\mu: f:\mathcal M\to\R, Lip(f)\leq 1 \right\}.
$$
The operator $\mathscr{S}$ is a contraction on the space $(\mathcal P(\mathcal M),\vertiii{\mbox{ }}).$ Indeed, for $\mu,\nu\in \mathcal P(\mathcal M)$ and a function $f:\mathcal M\to\R,$ we have that 
\begin{equation}\label{viva_Q}
\mathscr{S}(\mu)(f)-\mathscr{S}(\nu)(f)= \int \sum_{i=1}^n g_i(x)f(T_i(x))(d\mu-d\nu)(x).
\end{equation}
If $g:\mathcal M\to (0,1)$ and $T:\mathcal M\to \mathcal M$ with $Lip(T)<\infty,$ then 
$$
\begin{aligned}
&\sup\left\{\int f(T(x))g(x) d\mu (x): f:\mathcal M\to\R, Lip(f)\leq 1 \right\}\\ 
&\leq Lip(T) \|g\|\sup\left\{\int fd\mu (x): f:\mathcal M\to\R, Lip(f)\leq 1 \right\}.
\end{aligned}
$$
From Equation (\ref{viva_Q}) and the last observation we conclude that 
$$
\vertiii{\mathscr{S}(\mu)-\mathscr{S}(\nu)} \leq  \left(\sum_{i=1}^n Lip(T_i)\|g_i\|\right)  \vertiii{\mu-\nu} = L \vertiii{\mu-\nu},
$$
where $L=\sum_{i=1}^n Lip(T_i)\|g_i\|<1$ by hypothesis, and thus $\mathscr{S}$ is a contraction. On the other hand, $\mathcal P(\mathcal M)$ with the metric $\vertiii{ }$ is a complete metric space. It follows that $\mathscr{S}$ has a unique fixed point on $\mathcal P(\mathcal M)$ by the contraction mapping principle.
\end{proof}

\begin{remark}
A complete proof of the fact that $\mathcal P(\mathcal M)$ with the metric $\vertiii{ }$ is a complete metric space can be found in \cite{KantorovichBook}, Chapter 8, $\S$4, where it is proved that $(\mathcal P(\mathcal M),\vertiii{ })$ is a compact metric space. A more general result can be found in \cite{Kravchenko}, Theorem 4.2. On the other hand, it is also possible to prove the completeness of $\mathcal P(\mathcal M)$ with the metric $\vertiii{ }$ by using similar arguments than in \cite{U_Mosco}.   
\end{remark}

\subsection{Projection map and thermodynamic formalism}\label{sec_ThermForm_11Mar_2016}

To introduce our setting we need to define the metric space 
$$\mathcal{X} := \left\{ x = (x_n)_{n=0}^\infty   \hbox{ : } x_n \in \{ 1 ,\ldots, k\} , n \in \N_0\right\}=\{ 1 ,\ldots, k\} ^{\N_0}$$
with the metric 
$$
d(x, y) := \sum_{n=0}^\infty \frac{1 - \delta_{\{x_n\}}(y_n)}{2^n}.
$$
We consider $\mathcal{X}$ with the action of the shift $\sigma:\mathcal{X}\to \mathcal{X},$ defined by $(\sigma(x))_n=x_{n+1}$ for $n\in \mathbb N,$ where $x = (x_n)_{n=0}^\infty\in \mathcal{X}.$ The space $\mathcal{X}$ with the shift action is called a shift space.

\begin{definition}[Projection map]
Let $\mathcal T=\{T_i\}_{i=1}^{n}$ be an IFS on the unit interval. We define the projection map $\pi:\mathcal{X}\to [0,1]$ by
$$
\pi(x)=\pi_{\mathcal T}(x):=\lim_{n\to\infty}T_{x_0}\circ T_{x_1} \circ \cdots \circ T_{x_n}(0),
$$
where $x=(x_i)_{i=0}^{\infty}.$
\end{definition}

We recall some results on thermodynamic formalism and in particular we define the pressure function, Gibbs measures and the transfer operator. They will be useful in the proofs of the main theorems. We begin with the definition of the space of $\alpha$-H\"older functions.

\begin{definition}\label{Dec_12_Dec_HolderNorms}
Given $0<\alpha<1,$ let $\mathcal{C}^\alpha(\mathcal{X}, \mathbb R)$ denote the Banach space
of $\alpha$-H\"older continuous functions (or simply $\alpha$-H\"older functions) $f: \mathcal{X} \to \mathbb R$
with norm
$$
\|f\| := \max\{\|f\|_{\alpha},K \|f\|_{\infty}\},
$$
where
$$
\|f\|_{\alpha}:= \sup_{x \neq y} \left\{\frac{|f(x) - f(y)|}{d(x,y)^\alpha} \right\} \hbox{ and }
\|f\|_{\infty}:= \sup_{x} \{|f(x)| \}
$$ 
and $K>0$ is a constant.
\end{definition}

We now define the pressure function.

\begin{definition}\label{def_pressure_4_jul_2015}
Let $P: \mathcal{C}^{\alpha}(\mathcal{X}, \mathbb R) \to \mathbb R$ denote the pressure defined by
$$
P(\varphi) = \lim_{n \to +\infty} \frac{1}{n}\log \left( \sum_{\sigma^nx=x} \exp\left(\sum_{k=0}^{n-1}\varphi(\sigma^kx) \right)\right)
$$
where $\varphi \in \mathcal{C}^{\alpha}(\mathcal{X}, \mathbb R)$.
\end{definition}
The basic properties can be found in \cite{Rufus}, \cite{PP}, for example.
The following result gives an alternative definition of the pressure.

\begin{lemma}[Variational principle]
We can write 
$$
P(\varphi) = \sup\left\{  h(\nu) + \int \varphi d\nu
\hbox{ : } \nu \hbox{ is $\sigma$ invariant probability measure}
\right\},
$$
where $h(\nu)$ is the measure theoretic entropy with respect to $\nu.$
Moreover,  there is a unique $\sigma$ invariant probability measure $\mu_{\varphi}$ on $\B_{\mathcal{X}}$ which satisfies 
$P(\varphi) = h(\mu_{\varphi}) + \int \varphi d\mu_{\varphi}$. 
\end{lemma}

This leads to the following definition.

\begin{definition}
The measure $\mu_{\varphi}$ is called the Gibbs measure (or equilibrium state) 
for $\varphi \in \mathcal{C}^\alpha(\mathcal{X}, \mathbb R)$.
\end{definition}

The basic  properties  of the pressure  function we need are the following.

\begin{lemma}\label{lem_22Mar2016}
The  function $P: \mathcal{C}^\alpha(\mathcal{X}, \mathbb R) \to \mathbb R$ is analytic.  Moreover, the first and second derivatives are given by: 
\begin{enumerate}
\item
$\frac{d  P(\varphi + t\psi )}{d t} |_{t=0}= \int \psi d\mu_{\varphi}$; and 
\item
$\frac{\partial^2 P(\varphi + t_1\psi + t_2\xi)}{\partial t_1 \partial t_2} |_{(0,0)}= \sigma^{2}_{\mu_{\varphi}}(\psi, \xi)$ where 
$ \sigma^{2}_{\mu_{\varphi}}(\psi, \xi)$ is the  variance of  $\mu_{\varphi}$
\end{enumerate}
and $\psi , \xi \in \mathcal{C}^\alpha(\mathcal{X}, \mathbb R)$.
\end{lemma}

This result can be found in \cite{ruelle_book_TF} or \cite{PP}. For a proof including the details see \cite{MZbook}, Propositions 6.12 and 6.13 in Section 6.6.\\ 

Now we proceed to the definition of the Transfer operator.

\begin{definition}[Transfer operator]
Let $\mathcal T=\{T_i\}_{i=1}^n$ be an IFS on the unit interval with weight functions $\mathcal G=\{g_i\}_{i=1}^n$ and let $\psi \in \mathcal{C}^\alpha(\mathcal{X}, \mathbb R)$ be a H\"older function. We define the transfer operator $\mathscr{L}_{\psi} : \mathcal{C}^\alpha(\mathcal{X}, \mathbb R) \to \mathcal{C}^\alpha(\mathcal{X}, \mathbb R)$ by
$$
\mathscr{L}_{\psi}  w(x) = \sum_{\sigma y = x} e^{\psi(y)} w(y)
\hbox{ where } w \in \mathcal{C}^\alpha(\mathcal{X}, \mathbb R).
$$
\end{definition}

\subsection{Hausdorff Dimension}\label{sec_HD_7Mar_2016}

The notion of Hausdorff dimension allows to measure Borel sets in $\R^n$ associating to them a real number. This number is particularly useful to study fractal geometry, however in many cases hard to calculate. A complete discussion of the Hausdorff dimension of a set can be found in \cite{Falconer}.

\begin{definition}
Let $\mathcal E$ be a Borel set in $\R^n.$ The Hausdoff dimension $HD(\mathcal E)$ of $\mathcal E$ is defined by
$$
HD(\mathcal E):=\inf\{\alpha>0: H_\alpha(\mathcal E)\}=0,
$$
where 
$$
H_\alpha(\mathcal E):=\lim_{\epsilon\to 0} \inf\left\{\sum_{i=1}^{\infty}diam(B_i)^{\alpha}:\{B_i\} \mbox{ covers }\mathcal E \mbox{ and } diam (B_i)\leq \epsilon\right\}.
$$
\end{definition}

In this paper we are concerned with the Hausdorff dimension of the limit set $\mathcal K$ of an IFSs $\mathcal T.$ In this case, Bowen \cite{HD_Bowen} introduced a method relating the Hausdorff dimension $s$ of $\mathcal K$ with the solution of the equation $P(s\Phi)=0,$ where $P$ is the pressure function (Definition \ref{def_pressure_4_jul_2015}) and $\Phi$ is an appropriate function that depends on $\mathcal{T}.$ Some memorable references for applications of this approach are \cite{Ruelle_82_HD},\cite{Ruelle_HD_83},\cite{Manning_HD}, \cite{Mane90}.\\

The definition of Hausdorff dimension of a measure that we will use was introduced by Young in \cite{Young_82_Dim}.

\begin{definition}[Hausdorff dimension of $\mu$] Let $\mu$ be a Borel probability measure on $\R^d$ with bounded support. The Hausdorff dimension $HD(\mu)$ of $\mu$ is defined by
$$HD(\mu):=\inf\{HD(\mathcal E): \mu(\R^d\setminus \mathcal E)=0\}.$$
\end{definition}

Young proved the following theorem. 

\begin{theorem}[\cite{Young_82_Dim}]
Let $\mu$ be a Borel probability measure on $\R^d$ with bounded support. If 
\begin{equation}\label{Young_condition}
\lim_{\delta\to 0^+}\frac{\log \mu(B_{\delta}(x))}{\log \delta}=\alpha \mbox{ for }\mu-\mbox{a.e.}x\in supp(\mu),
\end{equation}
then $HD(\mu)=\alpha.$
\end{theorem}

A Borel probability measure on $\R^d$ satisfying the condition (\ref{Young_condition}) is called dimensional exact measure. In this paper we will study the Hausdorff dimension of the stationary measure of an IFS which satisfy the Open Set Condition. In this case, the stationary measure is well known to be a dimensional exact measure \cite{GDMS2003}. Moreover, we have the following theorem.

\begin{theorem}[Volume Lemma]\label{Volume_lemma}
Let $\mu$ be the stationary measure of an IFS on the unit interval which satisfy the Open Set Condition. Let $\nu$ be the unique probability measure on the shift space $\mathcal X$ such that $\pi_*\nu:=\nu\circ\pi^{-1} =\mu,$ where $\pi:\mathcal X\to [0,1]$ is the projection map.
If $$\chi_{\nu}:=\int -\log |dT_{x_0}(\pi(\sigma x))| d\nu(x)<\infty,$$ then 
$$
HD(\mu)=\frac{h_{\nu}(\sigma)}{\chi_\nu}.
$$
\end{theorem}

This theorem is from \cite{GDMS2003}. Earlier versions of it under stronger conditions can be find in \cite{Ledrappier81,Geronimo89}.

\section{Proofs}\label{PES_prel}
The main goal of this part is to prove Theorem \ref{main_theorem}, from it, we will deduce the other results. We have divided this section into six subsections. In the first, we study composition of functions, we settle some of our notation and we show some results on composition operators required in our proof. The results in this subsection follow from \cite{DLLLave}. In the second, we study the projection map, in particular we prove a useful result for the smoothness of projection map. Indeed, we prove that $\mathcal I_\epsilon\ni \lambda \mapsto \pi^{(\lambda)}\in \mathcal{C}^{\alpha}(\mathcal{X},\mathbb R)$ is $\mathcal{C}^{m-1},$ where $(\mathcal{X},\sigma)$ is a subshift of finite type. In the third we use some basic thermodynamic formalism results that we apply in the following subsections. In the fourth, we prove Theorem \ref{main_theorem}. In the fifth, we prove Theorem \ref{second_main_theorem}. Finally, in the sixth, we prove Theorem \ref{third_main_theorem}.

\subsection{First requirement: composition of functions}\label{sec:comp}

We will  use   results on composition of functions which are related  to those  in \cite{DLLLave}.  For the first part of the proofs, we do not really need to work with the full composition operator, whose definition depends on further smoothing conditions of its domain, but with a  simpler map whose definition only depends on the space $\mathcal{C}^{\alpha}(\mathcal{X},\mathbb R).$

\begin{definition}
Given a function $v:[0,1]\to\mathbb R,$ we define the map 
$$
\begin{aligned}
&v_*:&\mathcal{C}^{\alpha}(\mathcal{X},[0,1])&\to&& \mathcal{C}^{\alpha}(\mathcal{X},\mathbb R)\cr
&&f &\mapsto&& v_*(f):=v\circ f. 
\end{aligned}
$$
\end{definition}

Most of the results in this section deal with the regularity of the map $v_*.$ In order to state them precisely, we need to introduce the spaces of functions $\mathcal{C}^{n+\delta}([0,1],\mathbb R),$ for $0<\delta<1$ and $n>0$, which  correspond to the classic spaces of $n$ times continuously differentiable functions with the $n$-th derivatives are $\delta$-H\"older. We define these spaces rigorously.

\begin{definition}
For each $i>0,$ we denote the $i$-th derivative of $v:[0,1]\to \mathbb R,$ when it exists, by $d^i v$ (where $d^0 v=v$).

Given $n>0$ and $0<\delta<1,$ 
the space $\mathcal{C}^{n+\delta}([0,1],\mathbb R)$ is defined to be the space of functions $v:[0,1]\to \mathbb R$ such that $v$ is $n$ times differentiable and 
$$
\|v\|_{\mathcal{C}^0}:=\sup_{\tilde{x}\in [0,1]} |v(\tilde{x})|<\infty, 
$$
$$
\|v\|_{\mathcal{C}^n}:=\max_{i \in \{  0,\ldots,n \}} \|d^i v\|_{\mathcal{C}^0}<\infty
$$
and 
$$
\|d^n v\|_{\mathcal{C}^\delta}:= \sup_{\tilde{x}\neq \tilde{y}}\frac{|d^n v(\tilde{x})-d^n v(\tilde{y})|}{|\tilde{x}-\tilde{y}|^{\delta}}<\infty.
$$
We endowed it with the norm 
$$
\|v\|_{\mathcal{C}^{n+\delta}}=\sup(\|d^n v\|_{\mathcal{C}^\delta},\|v\|_{\mathcal{C}^n}).
$$
\end{definition}

This is a Banach space and in the case $n\in\N$ we have that
$$
\|v\|_{\mathcal{C}^{n+\delta}}=\sup(\|v\|_{\mathcal{C}^0},\|dv\|_{\mathcal{C}^{n-1+\delta}}).
$$

\begin{remark}
Given an integer $n>0,$ any function $v\in \mathcal{C}^{n+1}([0,1],\mathbb R)$  has $i$-th Lipschitz derivative for $i=0,1,\ldots n,$ i.e.   
$$
\emph{Lip}( d^iv ):= \sup_{\tilde{x}\neq \tilde{y}}\frac{| d^iv(\tilde{x})- d^iv(\tilde{y})|}{|\tilde{x}-\tilde{y}|}<\infty,
$$
for $i\in \{ 0 ,\ldots, n\}.$

This implies that $\mathcal{C}^{n}([0,1],\mathbb R)\subset \mathcal{C}^{m+\delta}([0,1],\mathbb R),$ for every $0\leq\delta\leq 1$ and $m< n,$ because Lipschitz functions are automatically $\delta$-H\"older for $0<\delta\leq 1.$ 
\end{remark}

The following result is analogous  to the proof of Proposition 6.2, part ii.2) in \cite{DLLLave}. 

\begin{lemma}\label{lem1_15oct_2014}
If $v\in \mathcal{C}^{1+\delta}([0,1],\mathbb R),$ then the map $v_*$ is $\mathcal{C}^0.$
\end{lemma}

\begin{proof}
We can choose arbitrarily $f_1,f_2\in \mathcal{C}^{\alpha}(\mathcal{X},[0,1])$ and $x,y\in \mathcal{X}.$ We can then consider a path  $\gamma_1:[0,1]\to  [0,1]$  joining $f_1(x)$ and $f_1(y)$ defined by $\gamma_1(t)=(1-t)f_1(x)+tf_1(y)$ and a path $\gamma_2: [0,1]\to [0,1]$ joining $f_2(x)$ and $f_2(y),$ defined by $\gamma_2(t)=(1-t)f_2(x)+tf_2(y).$  
We then have the following inequalities 
$$
\begin{aligned}
&|v(f_1(x))- v(f_2(x))- v(f_1(y))+v(f_2(y))|\cr
&\leq \int_0^1| dv(\gamma_1(t))\frac{d \gamma_1}{dt}(t)- dv(\gamma_2(t))\frac{d \gamma_2}{dt}(t) |dt  \cr
&\leq\int_0^1| \left( dv(\gamma_1(t)) - dv(\gamma_2(t)) \right) \frac{d \gamma_1}{dt}(t)|dt+ \int_0^1|  dv(\gamma_2(t))\left( \frac{d \gamma_1}{dt}(t)-\frac{d \gamma_2}{dt}(t)\right)|dt\cr
&\leq \| v\|_{\mathcal{C}^{1+\delta}}(|f_2(x)-f_1(x)|\cr
&\quad+|f_2(y)-f_1(y)|)^{\delta}  |f_1(x)-f_1(y) | +\|v\|_{\mathcal{C}^1} |f_1(x)- f_2(x)- f_1(y)+f_2(y)|.
\end{aligned}
$$
In particular, dividing both sides of the inequality by $d(x,y)^{\alpha}$ and taking the supremum over the set $\{x,y: x,y\in \mathcal{X}, x\neq y\},$ we obtain 
\begin{equation}
\begin{aligned}
\label{28Oct_name_eq}
\|v_*( f_1)- v_*( f_2)\|_{\alpha}&=\sup_{x\neq y} \frac{| (v\circ f_1-v\circ f_2)(x)-(v\circ f_1-v\circ f_2)(y) |}{d(x,y)^{\alpha}} \cr
&\leq 2^{\delta}\| v\|_{\mathcal{C}^{1+\delta}}\|f_2-f_1\|_{\infty}^{\delta}  \|f_1\|_{\alpha} +\|v\|_{\mathcal{C}^1} \|f_1-f_2\|_{\alpha}.
\end{aligned}
\end{equation}
The result follows.
\end{proof}

The next lemma is similar to the proof of Proposition 6.7 in \cite{DLLLave}. In preparation,  we need to introduce some definitions of differentiable operators.

Let $\mathcal{E}, \mathcal F$ be Banach spaces with norms $\|\cdot\|_{\mathcal{E}}$ and $\|\cdot\|_{\mathcal F}$, respectively. 
We denote the space of bounded linear functions from $\mathcal{E}$ to $\mathcal F$ by $L(\mathcal{E},\mathcal F).$ Let $\mathcal{U}\subset \mathcal{E}$ be an open set. We recall that a  function $f:\mathcal U\to \mathcal F$ is Fr\'echet differentiable at $u\in \mathcal U$ if we can find a bounded linear function $df(u)$ such that 
$$
\lim_{\epsilon\to 0}\frac{\|f(u+\epsilon h)-f(u)-\epsilon df(u)h \|_{\mathcal F}}{\epsilon}=0
$$
for every $h\in \mathcal E$ and uniformly with respect to $h\in B_1(0):=\{y\in \mathcal E: \|y\|_{\mathcal E}<1\}.$
We say that $f$ is differentiable in $\mathcal U$ if $f$ is differentiable at every point $u\in \mathcal U.$ We say that $f$ is of class $\mathcal{C}^1$ if it is differentiable and the mapping $df:\mathcal U\to L(\mathcal E,\mathcal F),$ $u\mapsto df(u)$ is continuous for the topology induced by the norm. Inductively, we define $d^nf$ to be the differential of $d^{n-1}f$ and we say that a function $f$ is $\mathcal{C}^n$ ($n$ times continuously differentiable) if $df:\mathcal U\to L(\mathcal E,\mathcal F)$ is $(n-1)$ times continuously differentiable.

\begin{lemma}\label{lem1_17oct_2014}
If $v\in \mathcal{C}^{2+\delta}([0,1],\mathbb R),$ then $v_*$ is $\mathcal{C}^1$ and
for all $f,h\in \mathcal{C}^{\alpha}(\mathcal{X},[0,1])$
the derivative of  $v_*$ is given by $d(v_*)(f)(h)=(dv)_*(f) \cdot h.$
\end{lemma}

\begin{proof}
If $v\in \mathcal{C}^{2+\delta}([0,1],\mathbb R),$ then it has a $\mathcal{C}^{2+\delta}$ extension to an open neighbourhood of $[0,1],$ i.e. $v\in \mathcal{C}^{2+\delta}((-\epsilon_1,1+\epsilon_1),\mathbb R)$ for some $\epsilon_1>0.$ This  induces an extension of $v_*$ to $\mathcal{C}^{\alpha}(\mathcal{X},(-\epsilon_1,1+\epsilon_1)).$ Let $f\in \mathcal{C}^{\alpha}(\mathcal{X},[0,1])$ and $h \in \mathcal{C}^{\alpha}(\mathcal{X},\mathbb R).$ 

To complete the proof we will need two simple inequalities: choose $0<\epsilon_2<1$ sufficiently small such that $\max_{t\in [0,1]} \|f+t\epsilon_2 h\|_{\infty}<1+\epsilon_1,$ then  
\begin{equation}\label{19Nov_name_eq}
\int_0^1  \|dv \circ (f+t\epsilon_2 h)- dv\circ f\|_{\infty}dt \leq \|h\|(\|v \|_{\mathcal{C}^2} + 1)\epsilon_2^{\delta} 
\end{equation}
and 
\begin{equation}\label{8Dec_name_eq}
\|dv \circ (f+t\epsilon_2 h)- dv\circ f\|_{\alpha}\leq 2^{\delta} \|v\|_{\mathcal{C}^{2+\delta}} \|\epsilon_2 h\|_{\infty}^{\delta}\|f\|_{\alpha}+\|v\|_{\mathcal{C}^2}\|\epsilon_2 h\|_{\alpha}.
\end{equation}
To prove (\ref{19Nov_name_eq}), we use that for every $t\in [0,1]$ and  $x\in \mathcal{X}$
$$
\begin{aligned}
&\frac{|dv \circ (f(x)+t\epsilon_2 h(x))- dv\circ f(x) |} {\epsilon_2}\cr
&=\frac{|d^2v (f (x)) \cdot t\epsilon_2  h(x) +o(t\epsilon_2 h(x))|} {\epsilon_2}\cr
&\leq  |d^2v (f (x))| \cdot |h(x)|+| h(x)| \frac{o(\epsilon_2)}{\epsilon_2}\cr
&\leq \|h\|(\|v \|_{\mathcal{C}^2} + 1).
\end{aligned}
$$
To prove (\ref{8Dec_name_eq}) we notice that by definition 
$dv \circ (f+t\epsilon_2 h)- dv\circ f= (dv)_* (f+t\epsilon_2 h)- (dv)_*f$
and use inequality (\ref{28Oct_name_eq}) with $dv$ instead of $v,$ $f+t\epsilon_2 h$ instead of $f_1$ and $f$ instead of $f_2.$

Fix $0<\epsilon_2<1$ sufficiently small for equation (\ref{19Nov_name_eq}) to hold, then
$$
\begin{aligned}
&\frac{1}{\epsilon_2}\|v_*(f+\epsilon_2 h)-v_*(f)-\epsilon_2 (dv)_*(f)\cdot h \|_{\alpha}\cr
&=\frac{1}{\epsilon_2}\|v\circ(f+\epsilon_2 h)-v\circ f-\epsilon_2 (dv\circ f)\cdot h \|_{\alpha}\cr
&=\|\int_0^1 [dv \circ (f+t\epsilon_2 h)- dv\circ f] \cdot h dt\|_{\alpha}\cr
&\leq \|h\|_{\infty} \int_0^1  \|dv \circ (f+t\epsilon_2 h)- dv\circ f\|_{\alpha}dt \cr
&\quad+ \|h\|_{\alpha} \int_0^1  \|dv \circ (f+t\epsilon_2 h)- dv\circ f\|_{\infty}dt\cr
&\leq\left(2^{\delta} \|v\|_{\mathcal{C}^{2+\delta}} \|\epsilon_2 h\|_{\infty}^{\delta}\|f\|_{\alpha}+\|v\|_{\mathcal{C}^2}\|\epsilon_2 h\|_{\alpha}\right)+ \|h\|(\|v\|_{\mathcal{C}^{2}}+1) \epsilon_2^{\delta} \cr
&\leq (4\|v\|_{\mathcal{C}^{2+\delta}}\max\{\|f\|_{\alpha},1\}+ 1)\|h\|\epsilon_2^{\delta},
\end{aligned}
$$
which proves the second part of the lemma. We used inequalities (\ref{19Nov_name_eq}) and (\ref{8Dec_name_eq}) in the penultimate inequality. 

Now that we have the formula for the derivative of  $v_*:$ 
\begin{equation}\label{george_eq_num_dos}
d(v_*)(f)(h)=(dv)_*(f) \cdot h
\end{equation}
 for all $f,h\in \mathcal{C}^{\alpha}(\mathcal{X},[0,1]),$ we can prove that $v_*$ is $\mathcal{C}^1.$ For this, it is enough to show that $d(v_*)$ is continuous. From (\ref{george_eq_num_dos}) we can see that $d(v_*)$ corresponds to $(dv)_*$ followed by the continuous linear map 
$$
\begin{aligned}
\mathscr{L}:\mathcal{C}^{\alpha}(\mathcal{X},L(\mathbb R,\mathbb R))  &\to L(\mathcal{C}^{\alpha}(\mathcal{X},[0,1]), \mathcal{C}^{\alpha}(\mathcal{X},\mathbb R)), \cr
 \xi &\mapsto [\mathscr{L}(\xi): h\mapsto \xi\cdot h].
\end{aligned}
$$
 Thus  we have that $d(v_*)=\mathscr{L}\circ (dv)_*$ is continuous, since  $(dv)_*$ is continuous by Lemma \ref{lem1_15oct_2014}.
\end{proof}

The next corollary follows by induction.

\begin{corollary}\label{cor1_17oct}
If $v\in \mathcal{C}^{n+\delta}([0,1],\mathbb R)$ for some integer $n\in\N,$ and thus $v_*$ is $\mathcal{C}^{n-1}$, as required.  
\end{corollary}

\begin{proof}
The case $n=1$ is covered  by Lemma \ref{lem1_15oct_2014}. If the result holds for $n$ and $ v\in \mathcal{C}^{n+1+\delta}([0,1],\mathbb R),$ then $(dv)_*$ is $\mathcal{C}^{n-1}$ by the inductive hypothesis. We can use the same argument as in  the last lines of the proof of Lemma \ref{lem1_17oct_2014} to obtain that $d(v_*)=\mathscr{L}\circ (dv)_*,$ where 
$\mathscr{L}$ is a continuous linear map, then $d(v_*)$ is $\mathcal{C}^{n-1}.$ Therefore, by definition, $v_*$ is $\mathcal{C}^{n},$ which concludes the proof. 
\end{proof}

A simple argument based in the previous corollary gives the following result that we use to prove the smoothness of the stationary probability measure.

\begin{corollary}\label{cor1_21oct_2014}
Suppose that we have a family of maps $\{v_i\in \mathcal{C}^{n+\delta}([0,1],\mathbb R):i \in \{ 1 ,\ldots, k\} \}$ for some integer $n\in\N,$ and consider the map $ F: \mathcal{C}^{\alpha}(\mathcal{X},[0,1]) \to \mathcal{C}^{\alpha}(\mathcal{X},\mathbb R),$ defined\footnote{The notation $v_{x_0}(\Pi(\sigma x))$   denotes  $v_{i}(\Pi(\sigma x))$ if $x_0=i.$}  by $F(\Pi)(x):=v_{x_0}(\Pi(\sigma x)),$ where $\Pi\in \mathcal{C}^{\alpha}(\mathcal{X},[0,1])$ and $x\in \mathcal{X}.$ Then $F$ is $\mathcal{C}^{n-1}.$  Moreover, for all $f,h\in \mathcal{C}^{\alpha}(\mathcal{X},[0,1])$
the derivative of  $F$ is given by 
$$
d(F)(f)(h)(x)=(d(v_{x_0}))_*(f (\sigma x) )\cdot h(\sigma x) \mbox{ for }x\in \mathcal{X}.
$$ 
\end{corollary}

\begin{proof}
The map $l_1:\mathcal{C}^{\alpha}(\mathcal{X},[0,1])  \to  [\mathcal{C}^{\alpha}(\mathcal{X},\mathbb R)]^k,$ defined by 
$$l_1(\Pi(x)):= [v_1(\Pi(x)),\ldots,v_k(\Pi(x))] \in [\mathcal{C}^{\alpha}(\mathcal{X},\mathbb R)]^k$$
is $\mathcal{C}^{n-1}$ by Lemma  \ref{lem1_17oct_2014}, and the map $l_2 :[\mathcal{C}^{\alpha}(\mathcal{X},\mathbb R)]^k \to \mathcal{C}^{\alpha}(\mathcal{X},\mathbb R),$ defined by 
$$l_2([f_1(x),\ldots,f_k(x)])= f_{x_0}(\sigma x)$$
is linear and continuous.  It follows that  the map $F=l_2 \circ l_1$ is $\mathcal{C}^{n-1}$.

To prove the formula for the derivative of $F$ we can use the chain rule and the fact that $l_2$ is linear to deduce that $dF=l_2\circ dl_1$ and $dl_1=[d(v_1)_*,\ldots, d(v_k)_*].$ This together with the formula for $d(v_i)_*$ for $i\in \{ 1 ,\ldots, k\} $ in Lemma \ref{lem1_17oct_2014} concludes the proof.
\end{proof}

To prove the smoothness of the Hausdorff dimension of the support of the stationary measure we additionally need the following results, whose proofs are analogous to the proofs in \cite{DLLLave} combined with simple arguments similar to the used in this section.\\

\begin{definition}
Given $n>0$ and $0<\delta<1,$ we define the composition operator by
$$
\begin{aligned}
&Comp:&\mathcal{C}^{n+\delta}([0,1],\mathbb R)\times \mathcal{C}^{\alpha}(\mathcal{X},\mathbb R)&\to&& \mathcal{C}^{\alpha}(\mathcal{X},\mathbb R)\cr
&&(v,f) &\mapsto&& Comp(v,f):=v\circ f.  
\end{aligned}
$$
\end{definition}

\begin{prop}\label{prop1_30nov}
Given $n\in\N$ and $0<\delta<1,$ the composition operator $\text{Comp}:\mathcal{C}^{n+\delta}([0,1],\mathbb R)\times \mathcal{C}^{\alpha}(\mathcal{X},\mathbb R)\to \mathcal{C}^{\alpha}(\mathcal{X},\mathbb R)$
is $\mathcal{C}^{n-1}.$
\end{prop}

This leads to the following corollaries.  

\begin{corollary}
The map $[\mathcal{C}^{n+\delta}([0,1],\mathbb R)]^k\times \mathcal{C}^{\alpha}(\mathcal{X},\mathbb R) \ni ([v_1,\ldots,v_k],f)$ $\mapsto$ $v_{x_0}\circ f (x) \in \mathcal{C}^{\alpha}(\mathcal{X},\mathbb R)$
is $\mathcal{C}^{n-1}.$
\end{corollary}

\begin{corollary}\label{cor3_1Dec_2014}
Let $n\in\N,$ $0<\delta<1,$ $\epsilon>0$ and suppose that we have for each $\lambda\in \mathcal I_\epsilon$ a family of maps $\{v^{(\lambda)}_i\in \mathcal{C}^{n+\delta}([0,1],\mathbb R):i \in \{ 1 ,\ldots, k\} \}$ and a  map $f^{(\lambda)}\in \mathcal{C}^{\alpha}(\mathcal{X},\mathbb R).$ If the map $\mathcal I_\epsilon\ni\lambda$ $\mapsto$ $[v_1^{(\lambda)},\ldots,v_k^{(\lambda)}]\in$ $[\mathcal{C}^{n+\delta}([0,1],\mathbb R)]^k$ is $\mathcal{C}^{n_1}$ for some $n_1>0,$ and the map 
$\mathcal I_\epsilon\ni\lambda$ $\mapsto$  $f^{(\lambda)}$ $\in \mathcal{C}^{\alpha}(\mathcal{X},\mathbb R)$ is $\mathcal{C}^{n_2}$ for some $n_2>0,$ then the map 
$\mathcal I_\epsilon\ni \lambda $ $\mapsto$  $v^{(\lambda)}_{x_0}\circ f^{(\lambda)} (x)$ $\in \mathcal{C}^{\alpha}(\mathcal{X},\mathbb R)$ is $\mathcal{C}^{\min(n_1,n_2,n-1)}.$
\end{corollary}

\subsection{Second requirement: projection map}\label{sec:projection}

We will introduce a projection map $\pi^{(\lambda)}:\mathcal{X}\to [0,1]$ for $\lambda\in \mathcal I_\epsilon$ that will be essential to study the differentiability of the stationary measure.

\begin{definition}
For each $\lambda\in \mathcal I_\epsilon$ we define the projection map $\pi^{(\lambda)}:\mathcal{X}\to [0,1]$ by
$$
\pi^{(\lambda)}(x):=\lim_{n\to\infty}T_{x_0}^{(\lambda)}\circ T_{x_1}^{(\lambda)} \circ \cdots \circ T_{x_n}^{(\lambda)}(0),
$$
where $x=(x_i)_{i=0}^{\infty}.$
\end{definition}

The following result is easily seen.

\begin{lemma}
There exists $\alpha >0$ such that each individual map $\pi^{(\lambda)}: \mathcal{X} \to [0,1]$ is $\alpha$-H\"older continuous.
\end{lemma}

\begin{proof}
Define $a:=\max_{i\in \{ 1 ,\ldots, k\} } \sup_{\lambda\in \mathcal I_\epsilon} \{\|d T_i^{(\lambda)}\|_{\mathcal{C}^0} \}<1$ and $\alpha:=-\frac{\log (a)}{\log (2)}$.
 Suppose that $x,y\in \mathcal{X}$ and chose $n=n(x,y)$ such that $x_i=y_i$ for $i\leq n$ and  $x_{n+1}\neq y_{n+1},$ then 
$$
\begin{aligned}
&|\pi^{(\lambda)}(x)-\pi^{(\lambda)}(y)|\leq a^n= \frac{1}{2^{\alpha n}}\leq d(x,y)^{\alpha}.
\end{aligned}
$$
This completes the proof.
\end{proof}

To make further use of the functional analytic approach it helps to choose a specific Banach space of H\"older continuous functions.

\begin{remark}\label{Rem1_22Oct}
We are now at liberty to choose values of $\alpha$ and $K$ which are most 
  convenient for us in definition of H\"older norm on $\mathcal{X}$  (i.e., Definition \ref{Dec_12_Dec_HolderNorms}).
  Denote $\theta_0 :=  \|dT_1^{(0)}\|_{\mathcal{C}^0}$  and   then fix  a choice of  $\theta_0 < \theta < 1$.  We can then choose $0< \alpha < 1$ sufficiently small such that $2^\alpha \theta_0 <
  \frac{ \theta + \theta_0}{2}$.  
   Finally, let us choose $K >0$ sufficiently large such that 
   $$   
 \emph{Lip} (dT_1) \|\pi^{(0)}\|_\alpha \frac{2^\alpha}{K} < \theta - \theta_0$$
    where $\emph{Lip}(dT_1)$ is the Lipschitz constant of the derivative of the contraction $T_1.$
\end{remark}

We may now prove the main proposition in this section.

\begin{prop}\label{prop_main_20Oct}
Provided $\alpha>0$ is chosen sufficiently small, the map $\mathcal I_\epsilon\ni \lambda\mapsto \pi^{(\lambda)}\in \mathcal{C}^{\alpha}(\mathcal{X},\mathbb R)$ is $\mathcal{C}^{m-1}.$
\end{prop}

\begin{proof}
For each $\lambda \in (-\epsilon, \epsilon)$ we let $ R^{(\lambda)} :\mathcal{C}^\alpha(\mathcal{X},  \mathbb R) \to \mathcal{C}^{\alpha}(\mathcal{X}, \mathbb R)$ be defined by 
$$( R^{(\lambda)} \Pi )(x) := T_{x_0}^{(\lambda)}(\Pi (\sigma x)),$$
and we construct the map $F:\mathcal I_\epsilon\times \mathcal{C}^{\alpha}(\mathcal{X}, \mathbb R) \to  \mathcal{C}^{\alpha}(\mathcal{X}, \mathbb R)$ defined by 
$F\left(\lambda, \Pi\right)=\left(I- R^{(\lambda)}\right)(\Pi),$  where  $\Pi\in \mathcal{C}^{\alpha}(\mathcal{X}, \mathbb R).$ As usual $D_2 F (0,\pi^{(0)})$ denotes the partial derivative of $F$ with respect to the second coordinate and evaluated in $(0,\pi^{(0)})$, i.e. for $F(0,\cdot):\mathcal{C}^{\alpha}(\mathcal{X}, \mathbb R) \to  \mathcal{C}^{\alpha}(\mathcal{X}, \mathbb R)$ defined by $F(0,\cdot)(\Pi)=F(0,\Pi),$ we define $D_2 F(0,\pi^{(0)}):= d F(0,\cdot)(\pi^{(0)}).$\\

We begin with some preliminary observations.
\begin{enumerate}
\item
First observe that 
$\pi^{(\lambda)}$ is a fixed point, i.e., $ R^{(\lambda)} \pi^{(\lambda)} = \pi^{(\lambda)}.$
\item
We next observe that the family of maps $(-\epsilon, \epsilon) \times \mathcal{C}^\alpha(\mathcal{X}, \mathbb R)
\ni (\lambda, \Pi) \mapsto  R^{(\lambda)}(\Pi)\in \mathcal{C}^\alpha(\mathcal{X}, \mathbb R)$ is $\mathcal{C}^{m-1}$.  
Clearly it is  $\mathcal{C}^{m-1}$ in $\lambda,$ whilst it is $\mathcal{C}^{m-1}$ in $\Pi$ by Corollary \ref{cor1_21oct_2014}.

\item  $D_2 F(0,\pi^{(0)})$ is a  linear homeomorphism of $\mathcal{C}^{\alpha}(\mathcal{X},\mathbb R)$ onto $\mathcal{C}^{\alpha}(\mathcal{X},\mathbb R)$.  Moreover,   we will prove that $(I - D_2 (R^{(0)}\pi^{(0)}) )$ is invertible. We call 
$${\mathscr{R}}^{(0)}:= D_2({ R}^{(0)}\pi^{(0)}).$$

On $\Pi \in \mathcal{C}^\alpha(\mathcal{X}, \mathbb R),$ $\mathscr{R}^{(0)}$ is given by  
$$\mathscr{R}^{(0)}(\Pi)(x)= d T_{x_0}^{(0)}\left( \pi^{(0)}(\sigma x)\right) \cdot \Pi(\sigma x), x\in \mathcal{X},$$
and this is clear using Corollary \ref{cor1_21oct_2014}.
Since each $T_i$ is a contraction  it is easy to see  that
$ \mathscr{R}^{(0)}:  \mathcal{C}^{0}(\mathcal{X}, \mathbb R) \to \mathcal{C}^{0}(\mathcal{X}, \mathbb R)$
satisfies  $\| \mathscr{R}^{(0)} \|_\infty < 1,$ i.e. $\mathscr{R}^{(0)}$ is a contraction on $\mathcal{C}^0.$  Using Remark \ref{Rem1_22Oct}  we will prove that 
$ \mathscr{R}^{(0)}$
is also a contraction on $\mathcal{C}^{\alpha}(\mathcal{X}, \mathbb R).$ For this, 
assume $\|\Pi\| \leq 1$ (and thus, in particular,  
$\|\Pi\|_\alpha \leq 1$ and $ \|\Pi\|_\infty \leq 1/K$).
We  can then use the triangle inequality to bound
$$
\begin{aligned}
&|\mathscr{R}^{(0)}(\Pi)(x) - \mathscr{R}^{(0)}(\Pi)(y)|\cr
&= \left| d T_{x_0}^{(0)}\left( \pi^{(0)}(\sigma x)\right) \Pi(x ) - d T_{y_0}^{(0)}\left( \pi(\sigma y)\right) \Pi( y )  \right|\cr
&\leq
\left| 
d T_{x_0}^{(0)}
\left( \pi^{(0)}(\sigma x)\right) \Pi( \sigma x ) -
 d T_{x_0}^{(0)}
 \left( \pi^{(0)}(\sigma x)\right)
  \Pi( \sigma y )  \right| \cr
&\qquad+\left| d T_{x_0}^{(0)}\left( \pi^{(0)}(\sigma x)\right) \Pi( \sigma y ) - d T_{y_0}^{(0)}\left( \pi^{(0)}(\sigma y)\right) \Pi( \sigma y )  \right|\cr
&\leq  \| d T_1^{(0)} \|_{\mathcal{C}^0} \left| \Pi( \sigma x ) - \Pi( \sigma y )\right|
+ \left| d T_{x_0}^{(0)}\left( \pi^{(0)}(\sigma x)\right) - d T_{y_0}^{(0)}\left( \pi^{(0)}(\sigma y)\right)  \right|. \|\Pi \|_\infty\cr
&\leq  \| d T_1^{(0)} \|_{\mathcal{C}^0} \|\Pi\|_\alpha d(\sigma  x,\sigma  y)^\alpha
+ Lip( d T_1^{(0)} ) |\pi^{(0)}(\sigma x) - \pi^{(0)}(\sigma y)| \frac{1}{K}
\cr
&\leq   \left( 2^{\alpha}   \| d T_1^{(0)} \|_{\mathcal{C}^0}\right)   d( x,  y)^{\alpha}
+\left( Lip( d T_1^{(0)}) \|\pi^{(0)} \|_{\alpha} 
\frac{2^{\alpha}}{K} \right)d(x,y)^{\alpha}\cr
&\leq \theta d(x,y)^{\alpha}, 
\end{aligned}
$$
where we have used Remark \ref{Rem1_22Oct}  in the last inequality. 
This implies $\| \mathscr{R}^{(0)} \|_{\alpha} < 1$.
\end{enumerate}

To end the proof we will use the implicit function theorem for Banach spaces (see for example \cite{AFBS65}).
The map $F$ is $\mathcal{C}^{m-1}$ in a neighbourhood of $(0, \pi^{(0)})$ of $\mathcal I_\epsilon\times \mathcal{C}^\alpha(\mathcal{X}, \mathbb R)$ and since $\max\{\|\mathscr{R}^{(0)}\|_{\infty}, \|\mathscr{R}^{(0)}\|_{\alpha} \}  < 1$ we see that $D_2F(0, \pi^{(0)}) = I - \mathscr{R}^{(0)}$ is invertible.  
Thus the hypotheses of the implicit function theorem are satisfied and the result follows.  
\end{proof}

\begin{example}
If $T_0 (x)=  \lambda x,$ $T_1 (x)=  \lambda x + t$ and $\mathcal{X} = \{0,1\}^{\N_0},$ then we can explicitly write 
 the map $\pi: \mathcal{X} \to \mathbb R$ as an infinite  series:
$$
\pi\left((x_n)_{n=0}^\infty  \right)=t  \sum_{n=0}^\infty \lambda^n x_n. 
$$
\end{example}

\subsection{Third requirement: thermodynamic formalism}\label{sec:thermodynamic}
We can deduce by classical techniques and an argument based in composition of operators the differentiability of a Gibbs measure that we will relate with the stationary measure using the projection maps. Also, we relate the Hausdorff dimension with the zero of $t\mapsto P(-t \Phi)$ by Bowen's method for some appropriate function $\Phi.$ This will be use to deduce the differentiability of the Hausdorff dimension.\\

In this subsection we consider an IFS $\mathcal T^{(\lambda)}=\{T_i^{(\lambda)}\}_{i=1}^n$ for $\lambda\in \mathcal I_\epsilon:=(-\epsilon,\epsilon) $and the family $\mathcal G^{(\theta)}$ of weights $\mathcal G^{(\theta)}=\{g_i^{(\theta)}\}_{i=1}^k$ for  $\theta\in \mathcal I_\epsilon.$  We associate a H\"older continuous function $\psi^{(\lambda,\theta)} \in \mathcal{C}^\alpha(\mathcal{X}, \mathbb R)$ defined 
 by 
 $$\psi^{(\lambda,\theta)}(x) := \log\left(g^{(\theta)}_{x_0}(\pi^{(\lambda)}(\sigma x))\right).$$ 
 
\begin{remark}
We see from the definition of $\mathscr{L}_{\psi^{(\lambda,\theta)}} $ and the property that $\sum_{i=1}^k g^{(\theta)}_i =1$ that 
$\mathscr{L}_{\psi^{(\lambda,\theta)}}  1 =1$, i.e., $\mathscr{L}_{\psi^{(\lambda,\theta)}}$ preserves the constant functions.
\end{remark}

We next recall the following classical result.

\begin{theorem}[Ruelle Operator Theorem]
There exists a maximal positive simple isolated eigenvalue $1$.   Moreover, 
\begin{enumerate}
\item there is a positive eigenvector $w_{\psi^{(\lambda,\theta)}}$, i.e.,  $\mathscr{L}_{\psi^{(\lambda,\theta)}}  w_{\psi^{(\lambda,\theta)}}  =w_{\psi^{(\lambda,\theta)}} $;
\item 
the equilibrium state  $\nu_{\psi^{(\lambda,\theta)}}$ is a fixed point for the dual operator, i.e.,
  $$\mathscr{L}_{\psi^{(\lambda,\theta)}}^* \nu_{\psi^{(\lambda,\theta)}} = \nu_{\psi^{(\lambda,\theta)}} $$ thus
   $\int  f d \nu_{\psi^{(\lambda,\theta)}}  = \int  ( \mathscr{L}_{\psi^{(\lambda,\theta)}}  f) d \nu_{\psi^{(\lambda,\theta)}}  $ for every continuous $f:\mathcal{X}\to \mathbb R.$
\end{enumerate}
\end{theorem}

\begin{proof}
The spectral  properties of the operator follow from the general results of Ruelle for transfer operators 
with any H\"older continuous function \cite{Rufus}, \cite{ruelle_book_TF}.  In this particular case 
the fact that the maximal eigenvalue is $1$ and the corresponding eigen-distribution is the equilibrium state
follows from the property that $\mathscr{L}_{\psi^{(\lambda,\theta)}}  1 =1$ and   \cite{walters}, \cite{Ledrappier}.
\end{proof}

\subsection{Proof of Theorem \ref{main_theorem}}\label{sec:proof_main}

We need  to relate the Gibbs measure  to the stationary measure $\mu_{\lambda,\theta},$ recall its definition in (\ref{24_07_2015_time:1:52}).
The strategy  of the proof of  Theorem \ref{main_theorem} consists of  the following steps:
\begin{enumerate}
\item We construct a probability measure $\nu_{\lambda,\theta}$ on the Borel sets of $\mathcal{X}:=\{ 1 ,\ldots, k\} ^{\mathbb N}$ such that for $w\in \mathcal{C}^{s+\delta}([0,1],\mathbb R)$ we have 
\begin{equation}\label{ec:main_thm_scheme_proof}
\int_{\mathcal{X}} w\circ \pi^{\lambda}(x) d\nu_{\lambda,\theta}(x)=  \int_0^1 w(\tilde{x})  d\mu_{\lambda,\theta}(\tilde{x}),
\end{equation}
where $\pi^{(\lambda)}\in \mathcal{C}^{\alpha}(\mathcal{X},[0,1])$  for $\lambda\in\mathcal I_\epsilon.$ The probability measure $\nu_{\lambda,\theta}$ corresponds to the Gibbs measure of an explicitly constructed H\"older potential that depends on both $\mathcal{T}^{(\lambda)}$ and $\mathcal{G}^{(\theta)}.$
\item \label{scheme:proof:2}  We prove that $\mathcal{C}^{\alpha}(\mathcal{X},\mathbb R) \ni \Pi\mapsto w\circ \Pi\in \mathcal{C}^{\alpha}(\mathcal{X},\mathbb R)$ is $\mathcal{C}^{s-1}.$   To achieve this, we use an argument of composition of operators (following  de la Llave and Obaya) which requires $w\in \mathcal{C}^{s+\delta}([0,1],\mathbb R).$ 
\item A similar  argument  is used to show   that $\mathcal I_\epsilon\ni \lambda \mapsto \pi^{(\lambda)}\in \mathcal{C}^{\alpha}(\mathcal{X},\mathbb R)$ is $\mathcal{C}^{m-1}.$ In order to apply the result in this case we need to use that $\mathcal T^{(\lambda)}$ is a family of $\mathcal{C}^{m+\beta}$ functions. 
We use an argument based on  the implicit function theorem that requires the family $\mathcal T^{(\lambda)}$ to be contractions.
\item We use 
 a classical result about regularity of Gibbs measures to prove that $\mathcal I_\epsilon\ni \lambda \mapsto \nu_{\lambda,\theta} \in \mathcal{C}^{\alpha}(\mathcal{X},\mathbb R)^*$ is $\mathcal{C}^{l-1}.$
\item As a consequence of the previous parts, we have that the map $\mathcal I_\epsilon\ni\lambda \mapsto (\nu_{\lambda,\theta},w\circ \pi^{(\lambda)})\in \mathcal{C}^{\alpha}(\mathcal{X},\mathbb R)^*\times \mathcal{C}^{\alpha}(\mathcal{X},\mathbb R)$ is $\mathcal{C}^{\min (l,m,s)-1}.$ On the other hand, the map  $\mathcal{C}^{\alpha}(\mathcal{X},\mathbb R)^*\times \mathcal{C}^{\alpha}(\mathcal{X},\mathbb R) \ni (\nu_{\lambda,\theta},w\circ \pi^{(\lambda)})\mapsto \nu_{\lambda,\theta}(w\circ \pi^{(\lambda)})=\int_{\mathcal{X}} w\circ \pi^{(\lambda)}(x) d\nu_{\lambda,\theta}(x) \in \mathbb R$ is $\mathcal{C}^{\infty}.$ This, together with equation (\ref{ec:main_thm_scheme_proof}) concludes the proof.  
\end{enumerate}

Now we can show the following result.

\begin{lemma}\label{Lem14_8_2014} 
Consider the family $\mathcal G^{(\theta)}$
of weights  $g_j^{(\theta)}$ for $j=1, \cdots, k$ and $-\epsilon < \theta < \epsilon$.  
 Then
the stationary measure for $\mathcal T^{(\lambda)}$ and  $\mathcal G^{(\theta)}$
is the image of the eigen-distribution  $\nu_{\psi^{(\lambda,\theta)}}$ for $\psi^{(\lambda)}$, i.e., 
$(\pi^{(\lambda)})_* \nu_{\psi^{(\lambda,\theta)}} = \mu_{\lambda,\theta}$.
\end{lemma}

\begin{proof}
By the uniqueness of the stationary measure, it is enough for us to check that 
$$
\int f(\tilde{x}) d\left((\pi^{(\lambda)})_* \nu_{\psi^{(\lambda,\theta)}} \right)(\tilde{x}) = \sum_{i=1}^k \int g^{(\theta)}_i(\tilde{x}) f(T_i\tilde{x})d \left((\pi^{(\lambda)})_* \nu_{\psi^{(\lambda,\theta)}}\right) (\tilde{x})
$$
holds for any continuous $f:[0,1]\to \mathbb R$ and $\tilde x \in [0,1]$. 
A straightforward manipulation yields 
$$
\begin{aligned}
\sum_{i=1}^k \int g^{(\lambda)}_i(\tilde{x}) f(T_i\tilde{x})d \left((\pi^{(\lambda)})_* \nu_{\psi^{(\lambda,\theta)}}\right) (\tilde{x}) &=  \int \left(\sum_{y\in \sigma^{-1}x}e^{\psi^{(\lambda,\theta)}(y)}f(\pi^{(\lambda)} y)\right) d \nu_{\psi^{(\lambda,\theta)}}(x)  \cr
&= \int \mathscr{L}_{\psi^{(\lambda,\theta)}} (f\circ \pi^{(\lambda)})  (x)  d \nu_{\psi^{(\lambda,\theta)}}(x)  \cr
&= \int f\circ \pi^{(\lambda)}  (x)  d \nu_{\psi^{(\lambda,\theta)}}(x)  \cr
&= \int f (\tilde{x})  d \left((\pi^{(\lambda)})_* \nu_{\psi^{(\lambda,\theta)}}\right) (\tilde{x})  \cr
\end{aligned}
$$
for every continuous function $f:[0,1]\to \mathbb R,$ where we have used that  $ \mathscr{L}_{\psi^{(\lambda,\theta)}}^* (\nu_{\psi^{(\lambda,\theta)}})=\nu_{\psi^{(\lambda,\theta)}}.$
\end{proof}

\begin{lemma}\label{lem_17_mar_2017}
For fixed $\theta\in \mathcal I_\epsilon,$ the map $\mathcal I_\epsilon\ni\lambda\mapsto \psi^{(\lambda,\theta)}\in \mathcal{C}^{\alpha}(\mathcal{X},\mathbb R)$ is $\mathcal{C}^{\min(l,m)-1}.$
\end{lemma}

\begin{proof}
Consider $\theta\in \mathcal I_\epsilon$ fixed.
By Corollary \ref{cor1_21oct_2014} we have that  
$
\mathcal{C}^{\alpha}(\mathcal{X},\mathbb R)\ni \Pi\mapsto g^{(\theta)}_{x_0}(\Pi (\sigma x))\in \mathcal{C}^{\alpha}(\mathcal{X},\mathbb R)
$
is $\mathcal{C}^{l-1}$ and by Proposition \ref{prop_main_20Oct} the map $\mathcal I_\epsilon\ni \lambda\mapsto \pi^{(\lambda)}\in \mathcal{C}^{\alpha}(\mathcal{X},\mathbb R)$ is $\mathcal{C}^{m-1},$ then the map 
$
\mathcal I_\epsilon\ni\lambda\mapsto  g^{(\theta)}_{x_0}(\pi^{(\lambda)}(\sigma x))\in \mathcal{C}^{\alpha}(\mathcal{X},\mathbb R)
$
is $\mathcal{C}^{\min(m,n)-1}.$ This proves that the map 
$
\mathcal I_\epsilon\ni\lambda\mapsto  \psi^{(\lambda,\theta)}(x) = \log\left(g^{(\theta)}_{x_0}(\pi^{(\lambda)}(\sigma x))\right)\in \mathcal{C}^{\alpha}(\mathcal{X},\mathbb R)
$
is $\mathcal{C}^{\min(m,n)-1},$ which concludes the proof.
\end{proof}

\begin{lemma}\label{lem1_5Novt}
For fixed $\lambda\in \mathcal I_\epsilon,$ the map $\mathcal I_\epsilon\ni \theta\mapsto \psi^{(\lambda,\theta)}\in \mathcal{C}^{\alpha}(\mathcal{X},\mathbb R)$ is $\mathcal{C}^{r}.$
\end{lemma}

\begin{proof}
From the hypothesis on the family $\mathcal G ^{(\theta)}$ and the definition of $\psi^{(\lambda,\theta)}$
$$
\begin{aligned}
\psi^{(\lambda,\theta)}(x)&=\log\left(g^{(\theta)}_{x_0}(\pi^{(\lambda)}(\sigma x))\right)\cr
&=\log\left(g_{x_0}(\pi^{(\lambda)}(\sigma x))+\theta g_{x_0,1}(\pi^{(\lambda)}(\sigma x))+\cdots+ \theta^{r} g_{x_0,r}(\pi^{(\lambda)}(\sigma x)) + o(\theta^{r}) \right)\cr
&=: t(\theta)
\end{aligned}
$$
where
$t(\theta)=t(0)+dt(0)\theta+\frac{1}{2!}d^2t(0)\theta+\cdots+o(\theta^r),$ and where $d^it(0)\in \mathcal{C}^{\alpha}(\mathcal{X},\mathbb R)$ is given by 
$$d^it(0)(x)=\frac{p_i \left[g_{x_0}(\pi^{(\lambda)}(\sigma x)), g_{x_0,1}(\pi^{(\lambda)}(\sigma x)),\cdots, g_{x_0,i}(\pi^{(\lambda)}(\sigma x)) \right]}{g_{x_0}(\pi^{(\lambda)}(\sigma x))^i},$$
 where  $p_i$ ($i\in \{ 0 ,\ldots, r\}$) are polynomials.  
\end{proof}

Using standard analytic perturbation theory  (cf. \cite{ruelle_book_TF}) and the previous corollary we have the following.

\begin{corollary}\label{cor_dec8}
$\mbox{}$
\begin{enumerate}
\item For fixed $\theta \in \mathcal I_\epsilon,$ the map $(-\epsilon, \epsilon) \ni \lambda \to \nu_{\psi^{(\lambda,\theta)}} \in \mathcal{C}^\alpha(\mathcal{X}, \mathbb R)^*$ is $\mathcal{C}^{\min(l,m)-1}.$
\item For fixed $\lambda\in \mathcal I_\epsilon,$ the map $(-\epsilon, \epsilon) \ni \theta \to \nu_{\psi^{(\lambda,\theta)}} \in \mathcal{C}^\alpha(\mathcal{X}, \mathbb R)^*$ is $\mathcal{C}^{r}.$
\end{enumerate}
\end{corollary}

In particular, this implies the following.
\begin{corollary}\label{Cor4_22Oct}
Given a H\"older continuous function $f \in \mathcal{C}^\alpha(\mathcal{X},  \mathbb R).$
\begin{enumerate}
\item For any fixed $\theta\in (-\epsilon, \epsilon),$ the map $(-\epsilon, \epsilon) \ni \lambda \mapsto \int f d \nu_{\psi^{(\lambda,\theta)}} \in  \mathbb R$
is  $\mathcal{C}^{\min (l,m)-1}.$
\item For any fixed $\lambda\in (-\epsilon, \epsilon),$ the map $(-\epsilon, \epsilon) \ni \theta \mapsto \int f d \nu_{\psi^{(\lambda,\theta)}} \in  \mathbb R$
is  $\mathcal{C}^{r}.$
\end{enumerate}
\end{corollary}

We now turn to the proof of Theorem  \ref{main_theorem}.

\begin{proof}[Proof of Theorem \ref{main_theorem}]
There are two parts.
\begin{enumerate}
\item From Corollary \ref{cor1_17oct} we deduce that for $f\in \mathcal{C}^{s+\delta}([0,1],\mathbb R),$ the map $\mathcal{C}^{\alpha}(\mathcal{X},\mathbb R) \ni \Pi\mapsto f\circ \Pi \in \mathcal{C}^{\alpha}(\mathcal{X},\mathbb R)$ is $\mathcal{C}^{s-1}$ and we know from Proposition \ref{prop_main_20Oct} that $\mathcal I_\epsilon\ni \lambda\mapsto  \pi^{(\lambda)}\in \mathcal{C}^{\alpha}(\mathcal{X},\mathbb R)$ is $\mathcal{C}^{m-1},$ then the map $\mathcal I_\epsilon\ni \lambda\to f\circ \pi^{(\lambda)}\in \mathcal{C}^{\alpha}(\mathcal{X},\mathbb R)$ is $\mathcal{C}^{\min(s,m)-1}.$ 
Using Corollary \ref{cor_dec8} we have that $(-\epsilon, \epsilon) \ni \lambda \to \nu_{\psi^{(\lambda,\theta)}} \in \mathcal{C}^\alpha(\mathcal{X}, \mathbb R)^*$ is $\mathcal{C}^{\min(l,m)-1},$ therefore the map
$l_1:\mathcal I_\epsilon\to \mathcal{C}^{\alpha}(\mathcal{X},\mathbb R)\times \mathcal{C}^{\alpha}(\mathcal{X},\mathbb R)^*,$ defined by $l_1(\lambda)=(f\circ \pi^{(\lambda)},\nu_{\psi^{(\lambda,\theta)}})$ is $\mathcal{C}^{\min(l,m,s)-1}.$
We define the map $l_2:\mathcal{C}^{\alpha}(\mathcal{X},\mathbb R)\times \mathcal{C}^{\alpha}(\mathcal{X},\mathbb R)^*\to\mathbb R$ by $l_2(v,\nu)=\int v d\nu$ for $v\in \mathcal{C}^{\alpha}(\mathcal{X},\mathbb R)$ and $\nu\in \mathcal{C}^{\alpha}(\mathcal{X},\mathbb R)^*.$ The map $l_2$ is $\mathcal{C}^{\infty}.$

We consider the map $F:=l_2\circ l_1,$ so $F(\lambda)=\int f\circ \pi^{(\lambda)} d\nu_{\psi^{(\lambda,\theta)}}$ is $\mathcal{C}^{\min(l,m,s)-1}.$ Finally by Lemma \ref{Lem14_8_2014}, $\int f\circ \pi^{(\lambda)} d\nu_{\psi^{(\lambda,\theta)}} =\int f d\mu_{\lambda,\theta},$ which concludes the proof of part 1.

\item For $f\in \mathcal{C}^{1}([0,1],\mathbb R),$ $f\circ\pi^{(\lambda)} \in \mathcal{C}^{\alpha}(\mathcal{X},\mathbb R)$ and the map $l_3:\mathcal I_\epsilon\to \mathcal{C}^{\alpha}(\mathcal{X},\mathbb R)^*$ defined by $l_3(\theta)=\nu_{\psi^{(\lambda,\theta)}}$ is $\mathcal{C}^r$ by Corollary \ref{cor_dec8}. We consider the map $G:\mathcal I_\epsilon\to \mathbb R,$ defined by $G(\theta)=l_2(f\circ\pi^{(\lambda)},l_3(\theta)),$ where $l_2$ is defined in the part 1 of this proof. By Lemma \ref{Lem14_8_2014} we have $G(\theta)=\int f d\mu_{\lambda,\theta}$ and $G$ is $\mathcal{C}^r$ since  $l_3$ is $\mathcal{C}^r$ and $l_2$ is $\mathcal{C}^{\infty}$.   This finishes the proof.
\end{enumerate}
\end{proof}

\subsection{Proof of Theorem \ref{second_main_theorem}}

Consider an IFS $\mathcal T$ as in Definition \ref{DEFnumOne} such that the sets $T_i^{(\lambda)}[0,1]$ are pairwise  disjoint for $i\in \{ 1 ,\ldots, k\}.$ Recall the definition of the projection map $\pi^{(\lambda)}:\mathcal X\to\R$ and the definition of the pressure $P$ (Definition \ref{def_pressure_4_jul_2015}). It is well known that the Hausdorff dimension of the limit set $\mathcal K(\lambda),$ that we call by $HD(\mathcal K(\lambda)),$ corresponds to the unique $s\in [0,1]$ such that $P(s\psi^{(\lambda)})=0,$ where $\psi^{(\lambda)}:\mathcal X\to\R$ is defined by $\psi^{(\lambda)}(x):=\log |dT^{(\lambda)}_{x_0}(\pi^{(\lambda)}(\sigma x)) |.$ 

\begin{prop}\label{prop1_2_dec}
Independently of $\mathcal G^{(\theta)},$ there exists a unique $t=t_{\lambda}=dim_{H}(\mbox{supp } \mu_{\lambda,\theta})$ such that 
$$
P\left(-t \psi^{(\lambda)}(x) \right)=0.
$$ 
\end{prop}

We are interested in the differentiability of the map $\mathcal I_\epsilon\ni\lambda\mapsto t_{\lambda}\in \mathbb R.$ Using  Corollary \ref{cor3_1Dec_2014} we can prove the main proposition we need.

\begin{prop}\label{prop_17_mar_2016_945}
The map
$
\mathcal I_\epsilon\ni\lambda\mapsto  \psi^{(\lambda)}(x)\in \mathcal{C}^{\alpha}(\mathcal{X}, \mathbb R) 
$
is $\mathcal{C}^{m-2}.$
\end{prop}

We can now prove our second theorem.

\begin{proof}[Proof of Theorem \ref{second_main_theorem}]
Since $P:\mathcal{C}^{\alpha}(\mathcal{X},\mathbb R)\to \mathbb R$ is real analytic it follows that $\mathcal I_\epsilon\ni\lambda\mapsto t_{\lambda}\in \mathbb R$ is $\mathcal{C}^{m-2}$ and using Proposition \ref{prop1_2_dec} we conclude the proof of Theorem \ref{second_main_theorem}.   
\end{proof}

\subsection{Proof of Theorem \ref{third_main_theorem}}

We proceed to the proof of the theorem immediately, as we have already developed all the machinery necessary for the proof.
\begin{proof}
We consider the map $\psi^{(\lambda,\theta)}(x):=\log
\left(g_{x_0}^{(\theta)}(\pi^{(\lambda)}(\sigma x) )  \right)$ defined in Subsection \ref{sec:thermodynamic}. The unique probability measure $\nu$ such that $\pi^{(\lambda)}_*\nu=\mu_{\lambda,\theta}$ is $\nu=\nu_{\psi^{(\lambda,\theta)}}=:\nu_{\lambda,\theta}$ (see Lemma \ref{Lem14_8_2014}).\\

For what follows we choose $\lambda_0,\theta_0\in \mathcal I_{\epsilon}$ fixed. By Lemma \ref{lem_17_mar_2017} the map 
$$
\mathcal I_{\epsilon}\ni\lambda\mapsto \psi^{(\lambda,\theta_0)}\in \mathcal C^{\alpha}(\mathcal X, \R)
$$
is $\mathcal C^{\min(l,m)-1}.$ By Lemma \ref{lem1_5Novt} the map 
$$
\mathcal I_{\epsilon}\ni\theta\mapsto \psi^{(\lambda_0,\theta)}\in \mathcal C^{\alpha}(\mathcal X, \R)
$$
is $\mathcal C^{r}.$ Since $P:\mathcal{C}^{\alpha}(\mathcal{X},\mathbb R)\to \mathbb R$ is real analytic by Lemma \ref{lem_22Mar2016}, it follows that the map
$$
\mathcal I_{\epsilon}\ni\lambda\mapsto P\left(\psi^{(\lambda,\theta_0)}\right)\in \R
$$
is $\mathcal C^{\min(l,m)-1}$ and the map 
$$
\mathcal I_{\epsilon}\ni\theta\mapsto P\left(\psi^{(\lambda_0,\theta)}\right) \in \R
$$
is $\mathcal C^{r}.$ On the other hand, by Corollary \ref{cor_dec8} we have that the map 
$$
\mathcal I_{\epsilon}\ni\lambda\mapsto \nu_{\lambda,\theta_0}\in \mathcal C^{\alpha}(\mathcal X, \R)^*
$$
is $\mathcal C^{\min(l,m)-1}$ and the map 
$$
\mathcal I_{\epsilon}\ni\theta\mapsto \nu_{\lambda_0,\theta}\in \mathcal C^{\alpha}(\mathcal X, \R)^*
$$
is $\mathcal C^{r}.$\\

We use now similar arguments to those in the proof of Theorem \ref{main_theorem}. By the pervious paragraph the map $A: \mathcal I_{\epsilon}\to \mathcal C^{\alpha}(\mathcal X, \R)\times \mathcal C^{\alpha}(\mathcal X, \R)^*$  defined by
$A(\lambda)=(\psi^{(\lambda,\theta_0)},\nu_{\lambda,\theta_0})$ is $\mathcal C^{\min(l,m)-1}$ and the map 
$B: \mathcal I_{\epsilon}\to \mathcal C^{\alpha}(\mathcal X, \R)\times \mathcal C^{\alpha}(\mathcal X, \R)^*$  defined by $B(\theta)=(\psi^{(\lambda_0,\theta)},\nu_{\lambda_0,\theta})$ is $\mathcal C^{r}.$ We define the map $l:\mathcal C^{\alpha}(\mathcal X, \R)\times \mathcal C^{\alpha}(\mathcal X, \R)^*\to \R$ by $l(v,\eta)=\int v d\eta.$ The map $l$ is $\mathcal{C}^{\infty},$ therefore, the maps $F_A(\lambda):=l\circ A (\lambda)=\int \psi^{(\lambda_0,\theta)}d\nu_{\lambda_0,\theta} $ and $F_B(\theta):=l\circ B(\theta)=\int \psi^{(\lambda_0,\theta)}d\nu_{\lambda_0,\theta}$ are $\mathcal C^{\min(l,m)-1}$ and $\mathcal C^{r},$ respectively.\\

Using the variational principle we have that 
$P\left(\psi^{(\lambda,\theta)}\right)-\int \psi^{(\lambda,\theta)}d\nu_{\lambda,\theta}=h_{\nu_{\lambda,\theta}}(\sigma).$ Combining this with the results of the previous paragraphs we conclude that the map 
$$
\mathcal I_{\epsilon}\ni\lambda\mapsto P\left(\psi^{(\lambda,\theta_0)}\right)-F_A(\lambda)=h_{\nu_{\lambda,\theta_0}}(\sigma)
$$ is $\mathcal C^{\min(l,m)-1}$ and the map  
$$
\mathcal I_{\epsilon}\ni\theta\mapsto P\left(\psi^{(\lambda_0,\theta)}\right)-F_B(\theta)=h_{\nu_{\lambda_0,\theta}}(\sigma)
$$
is $\mathcal C^{r}.$\\

We use now Proposition \ref{prop_17_mar_2016_945} with $\psi^{(\lambda)}(x):=\log |dT^{(\lambda)}_{x_0}(\pi^{(\lambda)}(\sigma x)) |,$ so that the map
$
\mathcal I_\epsilon\ni\lambda\mapsto  \psi^{(\lambda)}(x)\in \mathcal{C}^{\alpha}(\mathcal{X}, \mathbb R) 
$
is $\mathcal{C}^{m-2}.$ This implies that 
the map $$\mathcal I_\epsilon\ni \lambda\mapsto l(\psi^{(\lambda)},\nu_{\lambda,\theta_0})=\chi_{\nu_{\lambda,\theta_0}}\in\R$$ is  $\mathcal C^{\min(l-1,m-2)}$ and the map $$\mathcal I_\epsilon\ni\theta\mapsto l(\psi^{(\lambda_0)},\nu_{\lambda_0,\theta})=\chi_{\nu_{\lambda_0,\theta}}$$ is  $\mathcal C^{r}.$\\

Finally, by (\ref{cond_17_Mar_2016}) we have the hypothesis of Theorem \ref{Volume_lemma}, so that $HD(\mu_{\lambda,\theta})=\frac{h_{\nu_{\lambda,\theta}}}{\chi_{\nu_{\lambda,\theta}}}.$ Combining this and the results of the previous two paragraphs we deduce that the map
$$
\mathcal I_\epsilon\ni\lambda\mapsto HD(\mu_{\lambda,\theta_0})\in\R^+  
$$
is $\mathcal{C}^{\min(l-1,m-2)}$ and the map 
$$
\mathcal I_\epsilon\ni\theta\mapsto HD(\mu_{\lambda_0,\theta})\in\R^+  
$$
 is  $\mathcal C^{r}.$
\end{proof}

\section{Examples}\label{sec_examples}

In this section we exhibit different examples of application of our main results.

\subsection{A simple example}

Let $T_1, T_2: \mathbb R \to \mathbb R$ be the affine maps $T_1(x) = \alpha x + \beta_1$ 
and $T_2(x) = \alpha x + \beta_2$ with $0< \alpha < 1$.
Let us consider the weights $p_1, p_2 >0$ with $p_1 + p_2 = 1$. 
The unique stationary probability measure
$\mu =  \mu_{\alpha, \beta_1, \beta_2, p_1, p_2}$  in this case is given by the limit in the weak topology
$$
\mu : =  \lim_{n\to+\infty} \sum_{i_1, \cdots, i_n \in \{1,2\} }p_{i_1} \cdots p_{i_n} \delta_{T_{i_1} \circ \cdots \circ T_{i_n}(0)}. 
$$

If we further assume for simplicity  that $\alpha=0.5$ and $\beta_1 = 0$, $\beta_2 =\alpha$
then the two images 
$T_1[0,1] = [0, \alpha]$, $T_2[0,1] = [ \alpha,1]$ partition the unit interval and 
$\mu$ will  be supported on the unit  interval.  
Finally, in this case it is simple to see that 
 $\mu$ is then the Lebesgue measure  if and only if  $p_1 = p_2=0.5.$

We can consider the dependence of the stationary measure on the parameters 
$\alpha, \beta_j$ and $p_j$ ($j=1,2$) which form  a two dimensional space.  
 For any $\mathcal{C}^{2+\delta}$ function $w: [0,1] \to \mathbb R$ (with $0<\delta\leq 1$) we then have that the map 
$$(0,1)   \ni \alpha  \mapsto\int w d \mu_{\alpha, p_1} \in \mathbb R, $$
is  $\mathcal{C}^1$, and 
$$ (0,1) \ni   p_1  \mapsto\int w d \mu_{\alpha, p_1} \in \mathbb R, $$
is $\mathcal{C}^{\infty}$,
where we write $\mu_{\alpha,  p_1} = \mu_{\alpha,  0, p_1, 1- p_1}.$ It is clear in this example that the Hausdorff dimension of the limit set and the Hausdorff dimension of the measure $\mu$ are both $\mathcal{C}^{\infty}.$

\subsection{A geometric example}

We present in Example \ref{example_classical_Schottky_group} a result on classical Schottky groups. Our machinery is however limited to the case of \emph{unique contraction} that we define in what follows.

\begin{definition}[Unique contraction]
Let $\Gamma \subset SL(2,\mathbb C)$ be a classical Schottky group and suppose that $\Gamma$ is generated by the M\"obius transformations $\{\gamma_i\}_{i=1}^k.$ For each $i\in \{ 1 ,\ldots, k\}$ define $\mathcal{U}_i:=\{z\in\mathbb C: | d\gamma_i(z)|<1\} \subset \mathbb C$ and call by $T_i$ the map $\gamma_i|_{\mathbb C\setminus \mathcal{U}_i}: \mathbb C\setminus \mathcal{U}_i\to \mathcal{U}_i.$  We say that $\Gamma$ has a unique contraction if $dT_i=dT_j$ for every $i,j\in \{2,\ldots, k\}.$
\end{definition}

 \begin{example}\label{example_classical_Schottky_group}
For $\lambda\in \mathcal I_\epsilon=(-\epsilon,\epsilon),$ let $\Gamma_\lambda \subset SL(2,\mathbb C)$ be a classical Schottky group such that $\Gamma_0$ has a unique contraction and $\mathcal I_\epsilon\ni \lambda\mapsto \Gamma^{\lambda}\in SL(2,\mathbb C)$ is $\mathcal{C}^m.$ 
Let $\mu_\lambda$ be the conformal probability measure that satisfies $$g^*\mu_\lambda=|dg|^{\mathscr{H}_{\lambda}}\mu_\lambda,$$ where $\mathscr{H}_{\lambda}=HD(\Lambda_{\lambda})$ is the Hausdorff dimension of the limit set $\Lambda_{\lambda}$ for $\Gamma_\lambda.$ If $w:\mathbb C\to \mathbb R$ is a compactly supported $\mathcal{C}^{s+\delta}$ function then the map $$\mathcal I_\epsilon\ni\lambda\mapsto \int f d\mu_\lambda$$ is $\mathcal{C}^{\min(m,s-1)}.$
\end{example}

\begin{proof}[Proof of Example \ref{example_classical_Schottky_group}]
Suppose that $\Gamma_\lambda$ is generated by some M\"obius transformations $\{\gamma^{\lambda}_i\}_{i=1}^k$ and for each $i\in \{ 1 ,\ldots, k\}  $ define $\mathcal{U}^\lambda_i:=\{z\in\mathbb C: | d\gamma^{\lambda}_i(z)|<1\} \subset \mathbb C.$ For each $i,j\in \{  1,\ldots, k\} $ call by $T^\lambda_i$ the map $\gamma^{\lambda}_i|_{\mathbb C\setminus \mathcal{U}_i}: \mathbb C\setminus \mathcal{U}_i\to \mathcal{U}_i$ and define the map $T^\lambda_{i,j}:  \mathcal{U}_i\to \mathcal{U}_j$ such that $T^\lambda_{i,j}=T^\lambda_j|_{\mathcal{U}_i}.$
We consider the shift space 
$$\Sigma:=\{x=(x_n)_{n=0}^{\infty}: x_n\in \{  1,\ldots,k \},x_{n}\neq x_{n+1},n\in\N_0\}\subset \{1,\ldots, k\}^{\N_0}$$
and define the projection map $\pi^{\lambda}:\Sigma\to \Lambda_\lambda\subset \mathbb C,$ by $x\mapsto \lim_{n\to\infty} T^\lambda_{x_0}T^\lambda_{x_1}\cdots T^\lambda_{x_n}(z_0)$ where $z_0\in \mathbb C$ is fixed and $\Lambda_\lambda:= \{\lim_{n\to\infty} T^\lambda_{x_0}T^\lambda_{x_1}\cdots T^\lambda_{x_n}(z_0):x\in \Sigma \}$ is the limit set for $\Gamma_\lambda$. We notice that $\pi^\lambda\in \mathcal C^\alpha(\Sigma,\mathbb C)$ for some small $\alpha>0.$ The conformal probability measure $\mu_\lambda$ satisfies that $\mu_\lambda=\pi^\lambda_* \mu^\lambda$ for $\mathcal L_\lambda^* \mu^\lambda=\mu^\lambda$ where
$$
\mathcal L_\lambda w(x)=\sum_{\begin{subarray}{c} y\in \sigma^{-1}x\\ y\in \Sigma \end{subarray}} |d T^\lambda_{y_0,x_0} (\pi^\lambda y) |^{\mathscr{H}_{\lambda}}w(T^\lambda_{y_0,x_0} (\pi^\lambda y)),w:\Sigma\to \mathbb R, x\in \Sigma.
$$ 
We know from \cite{Ruelle_82_HD} that the Hausdorff dimensions of the limit set for $\Gamma$ is a real analytic function on the deformation space of a Schottky group, then the map $\mathcal I_\epsilon\ni\lambda\mapsto \mathscr{H}_{\lambda} \in\mathbb R$ is $\mathcal C^m.$ On the other hand, the map $\mathcal I_\epsilon\ni\lambda\mapsto \pi^{\lambda}\in \mathcal C^{\alpha}(\Sigma,\mathbb C)$ is $\mathcal C^m$ (we can use the same proof of Proposition \ref{prop_main_20Oct}, the main difference is that now when applying Corollary \ref{cor1_21oct_2014} we obtain $\mathcal C^m$ and not $\mathcal C^{m-1}$ as the maps $T^\lambda_{i}$ are $\mathcal C^{\infty}$ and not just $\mathcal C^{m+\delta}$). Then the map $\mathcal I_\epsilon\ni\lambda\mapsto \mathscr{H}_{\lambda}\log | d T^\lambda_{y_0,x_0} (\pi^\lambda y) \pi^{\lambda}| \in \mathbb R$ is $\mathcal C^m$ and by perturbation theory so is the map $\mathcal I_\epsilon\ni \lambda\mapsto \mu^{\lambda}\in \mathcal C^{\alpha}(\Sigma,\mathbb R)^*.$ Finally, we have that for $w:\mathbb C\to \mathbb R$ a compactly supported $\mathcal C^{\infty}$ function $\int w\circ \pi^\lambda d\mu^\lambda=\int w d\mu_\lambda$ and therefore the map $\lambda \mapsto \int w d\mu_\lambda$ is $\mathcal C^{m}$ by an application of Corollary \ref{cor1_17oct}, which concludes the proof.
\end{proof}

\subsection{Some general examples}

A careful look at Theorem \ref{main_theorem} and to it proof allows to obtain  similar results to the ones showed in the introduction under much weaker hypotheses. This is the propose of this subsection. We start by modifying Definition \ref{DEFnumOne} and replacing it by:

\begin{definition}\label{03_07_2015_def}
Assume that $\delta,\epsilon\in (0,1)$ , $k,l,m,n,p \in \N\setminus\{1\},$ $q\in\N$ and let $\Lambda,\Theta$ be open intervals $\Lambda,\Theta \subset \R.$ 
\begin{enumerate}
\item Let $$\mathcal T=\mathcal T(\Lambda,k,l,m,\delta) :=\left\{ \{T_i^{(\lambda)}\}_{i=1}^k: \lambda\in \Lambda \right\}$$ be a family of contractions such that for $\lambda\in \Lambda$ and $i\in \{ 1 ,\ldots, k \}  :$
$$T_{i}^{(\lambda)}=\tilde{T}_i(\lambda,\cdot),$$
where 
\begin{enumerate}
\item $\tilde{T}_i(\lambda,\cdot) \in \mathcal{C}^{l+\delta}([0,1],[0,1]),$
\item  $\sup_{\lambda\in \Lambda} \| \frac{\partial}{\partial x} \tilde{T}_1(\lambda,\cdot) \|_{\mathcal C^0}<1,$
\item $\tilde{T}_1 (\cdot,\cdot)\in \mathcal C^{m}(\Lambda\times [0,1],[0,1]),$ and
\item $\frac{\partial}{\partial x} \tilde{T}_i(0,x)=\frac{\partial}{\partial x} \tilde{T}_j(0,x)$ for all $i,j.$
\end{enumerate}
\item  \label{limit_set_7_7_2015} On a family $\mathcal T$ for every $\lambda\in \Lambda,$
we define the limit set $\mathcal K(\lambda)$ as the unique non empty closed set $\mathcal K\subset [0,1]$ such that  
$$
\mathcal K=\cup_{i=1}^k T_i^{(\lambda)}\mathcal K.
$$ 

\item We define $(\mathcal T,\mathcal G),$
where $$\mathcal G= \mathcal G(\Theta, k,n,p,\epsilon) :=\left\{\left\{g_i^{(\theta)}\right\}_{i=1}^k:\theta\in \Theta\right\}$$ is a family of weight functions such that
\begin{enumerate}
\item $$\sum_{i=1}^k \left\| g_i^{(\theta)}\right\|_{\mathcal C^0}   \emph{Lip}\left(T^{(\lambda)}_i \right)<1 \mbox{ for all }\lambda\in\Lambda,\theta\in \Theta; $$ and
\item for every $\theta\in \Theta,  i\in \{ 1 ,\ldots, k \} :$
$$g_{i}^{(\theta)}=\tilde{g}_i(\theta)$$
where for some $\beta\in(0,1/2)$ we have 
\begin{enumerate}
\item $\tilde{g}_i(\theta)\in \mathcal C^{n+\epsilon}([0,1],\R^+),$
\item $\tilde{g}_i(\cdot)\in \mathcal{C}^{q}\left(\mathcal I, \mathcal C^{n+\epsilon}([0,1],\R^+)\right).$
\end{enumerate}
\end{enumerate}
\end{enumerate}
\end{definition}

If we do not consider the normalisation condition on the weight functions, we require a generalised definition of stationary measures. In order to deal with this we introduce the next definition.

\begin{definition}
Given the families $(\mathcal T, \mathcal G),$ define $h_i^{(\lambda,\theta)}:=\left(g_i^{(\theta)}\right)^{s^{\lambda,\theta}},$ where $s^{\lambda,\theta}\in [0,1]$ is unique solution of $P\left(s^{\lambda,\theta} \log\left( g_{x_0}^{(\theta)}(\pi^{(\lambda)}(\sigma x)\right)\right)=0$ and $P$ is the Pressure. A  generalized stationary measure  $\mu = \mu_{\lambda,\theta}$ is the unique probability measure on $[0,1]$ that satisfies 
$$\int f(x) d\mu (x) = \sum_{i=1}^k \int  h_i^{(\lambda,\theta)} (x) f(T_i^{(\lambda)}(x))d\mu(x),$$ for any continuous  function $f:[0,1] \to \mathbb R.$
\end{definition}

Under the hypotheses of Definition \ref{03_07_2015_def},  a step-by-step equal proof that the one given for Theorem \ref{main_theorem} gives us the following result:
 
 \begin{theorem}\label{teo1_13_jul_2014} Let fix $a\in (\N\setminus\{1\})\cup\{\infty\}$ and $\rho\in (0,1).$  On $(\mathcal T, \mathcal G),$ for the generalized stationary probability measure $\mu_{\lambda,\theta}$ with $\lambda\in\Lambda, \theta\in\Theta,$ or in the case $\Lambda=\Theta,$ for the generalized stationary probability measure $\mu_{\lambda,\lambda}=\mu_{\lambda}$ for $\lambda\in\Lambda,$ we have: 
\begin{enumerate}
\item For $\theta \in\Theta$ and $f\in \mathcal C^{a+\rho}(\hat{\mathcal{K}},\R),$ where $\hat{\mathcal{K}}\supset \cup_{\lambda\in\Lambda} \mathcal{K}(\lambda),$ the map $F:\Lambda\to\R$ defined by
$$
F(\lambda)=\int f d\mu_{\lambda,\theta}
$$
belongs to $\mathcal C^{r}(\Lambda,\R)$ with $r=\min\{l-1,m-1,a-1\}.$
\item 
For $\lambda \in\Lambda$ and $f\in \mathcal C^{1}(\hat{\mathcal{K}},\R),$ the map $F:\Theta\to\R$ defined by
$$
F(\theta)=\int f d\mu_{\lambda,\theta}
$$
belongs to $\mathcal C^{q}(\Theta,\R).$
\item For $\Lambda=\Theta$ and $f\in \mathcal C^{a+\rho}(\hat{\mathcal{K}},\R),$ the map $F:\Lambda\to\R$ defined by
$$
F(\lambda)=\int f d\mu_{\lambda}
$$
belongs to $\mathcal C^{r}(\Lambda,\R)$ with $r=\min\{l-1,m-1,a-1,n-1,q\}.$
 \end{enumerate}
\end{theorem}

An easy example of application of Theorem \ref{teo1_13_jul_2014} that Theorem  \ref{main_theorem} fails is the case that $x_0\in [0,1]\setminus \cup_{\lambda\in\Lambda} \mathcal K(\lambda)$ and $f(x)=|x-x_0|.$\\ 

We end this subsection with two examples. In the first we can apply our theorem and it is possible to experimentally see the regularity of the map $F(\lambda).$ In the second, the hypothesis on the smoothness of the contractions is not satisfied. In this case, experimentally the map $F(\lambda)$ looks $\mathcal C^0$ but not $\mathcal C^1,$ however we cannot prove it, as our method of composition of operator does not work. The first example is the following:

\begin{example}\label{ex1_16_07_2015}
 Let us consider $\Lambda=\Theta=[1/6,1/3],$ $x\in[0,1],n\in\N,\lambda\in\Lambda,$ 
$$
\begin{aligned}
\phi(x,n)&=x^{n+1}\sin(1/x)\in \mathcal C^{n}(\R,\R)\setminus \mathcal C^{n+1}(\R,\R),\\
T_1^{(\lambda)}(x)&=\lambda x+\phi(\lambda-0.25,3) +0.01,\\
T_2^{(\lambda)}(x)&=\lambda x+\frac{2}{3}+\phi(\lambda-0.25,3),\\
g_1^{(\lambda)}(x)&=\lambda\1_{[0,1/2)}(x)+(1-\lambda) \1_{[1/2,1]}(x),\\
g_2^{(\lambda)}(x)&=(1-\lambda)\1_{[0,1/2)}(x)+(\lambda)\1_{[1/2,1]}(x),\mbox{ and}\\
f(x)&=
\begin{cases}
-x & \mbox{if }x\in [0,1/2)\\
 x^2 & \mbox{if }x\in [1/2,1].
\end{cases}
\end{aligned}
$$
Then the map $F:\Lambda\to\R,$ defined by $F(\lambda)=\int f(x)d\mu_{\lambda}(x),$ belongs to  $C^1(\Lambda,\R).$ Moreover, for any interval $\Lambda'\subset [1/6,1/4)$ or $\Lambda'\subset (1/4,1/3],$ we have that $F|_{\Lambda'}\in C^{\infty}(\Lambda',\R).$

\begin{figure}[!]
\centering
\fbox{
\includegraphics[width=0.9\textwidth]{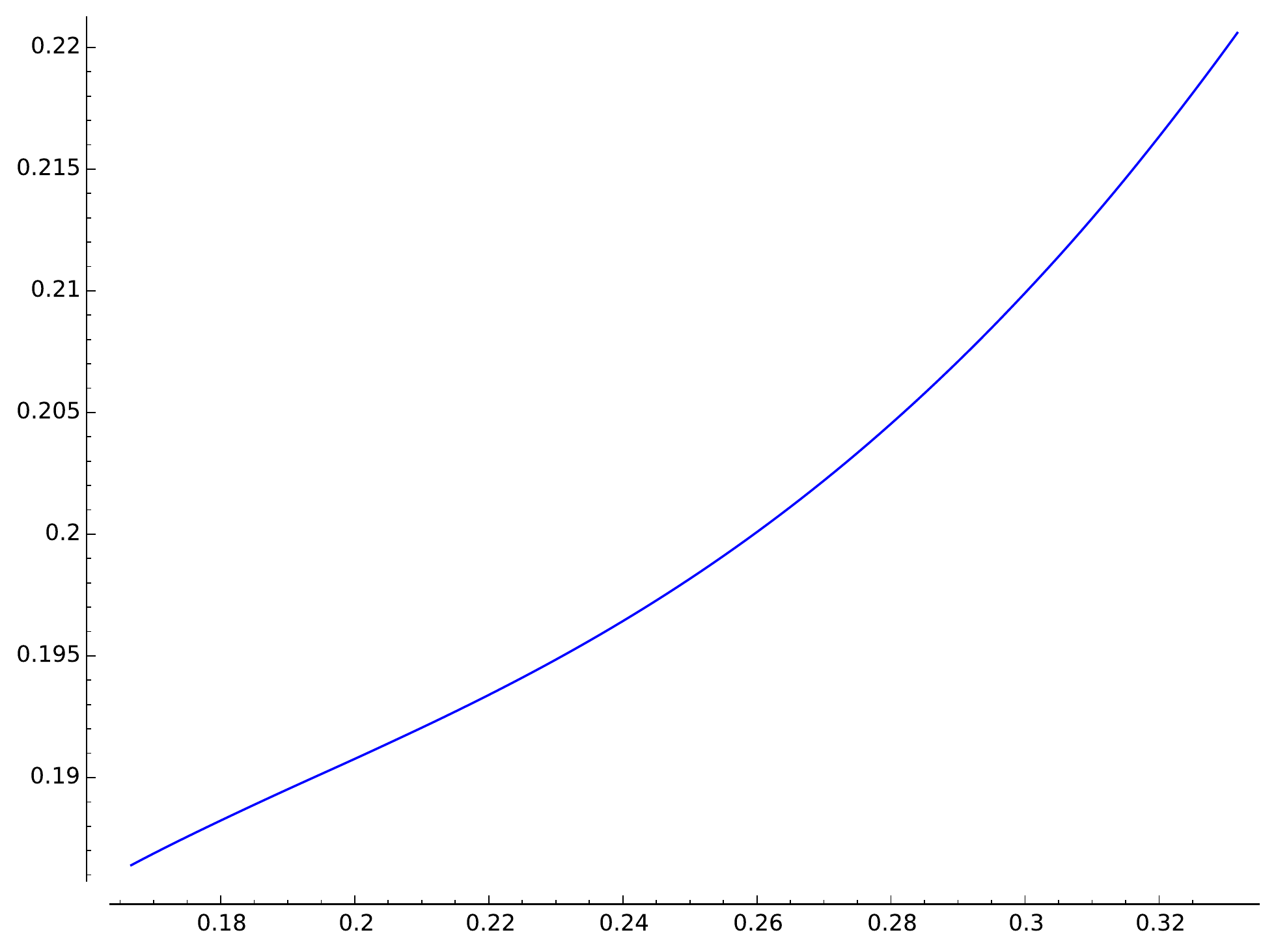}
}
\caption{Graph of $F:\Lambda\to\R$ in Example \ref{ex1_16_07_2015}}\label{figure_ex1_16_07_2015}
\end{figure}

\end{example}

The second example, where our results are not longer valid, is the following:

\begin{example}\label{ex2_16_07_2015}
 Let us consider $\Lambda=\Theta=[1/6,1/3],$ $x\in[0,1],n\in\N,\lambda\in\Lambda,$ 
$$
\begin{aligned}
T_1^{(\lambda)}(x)&=\lambda x+\phi(\lambda-0.25,1) +0.01,\\
T_2^{(\lambda)}(x)&=\lambda x+\frac{2}{3}+\phi(\lambda-0.25,1),\\
g_1^{(\lambda)}(x)&=\lambda\1_{[0,1/2)}(x)+(1-\lambda) \1_{[1/2,1]}(x),\\
g_2^{(\lambda)}(x)&=(1-\lambda)\1_{[0,1/2)}(x)+(\lambda)\1_{[1/2,1]}(x),\mbox{ and}\\
f(x)&=
\begin{cases}
-x & \mbox{if }x\in [0,1/2)\\
 x^2 & \mbox{if }x\in [1/2,1].
\end{cases}
\end{aligned}
$$
Does the map $F:\Lambda\to\R,$ defined by $F(\lambda)=\int f(x)d\mu_{\lambda}(x),$ belongs to  $C^0(\Lambda,\R)?$

\begin{figure}[!]
\centering
\fbox{
\includegraphics[width=0.9\textwidth]{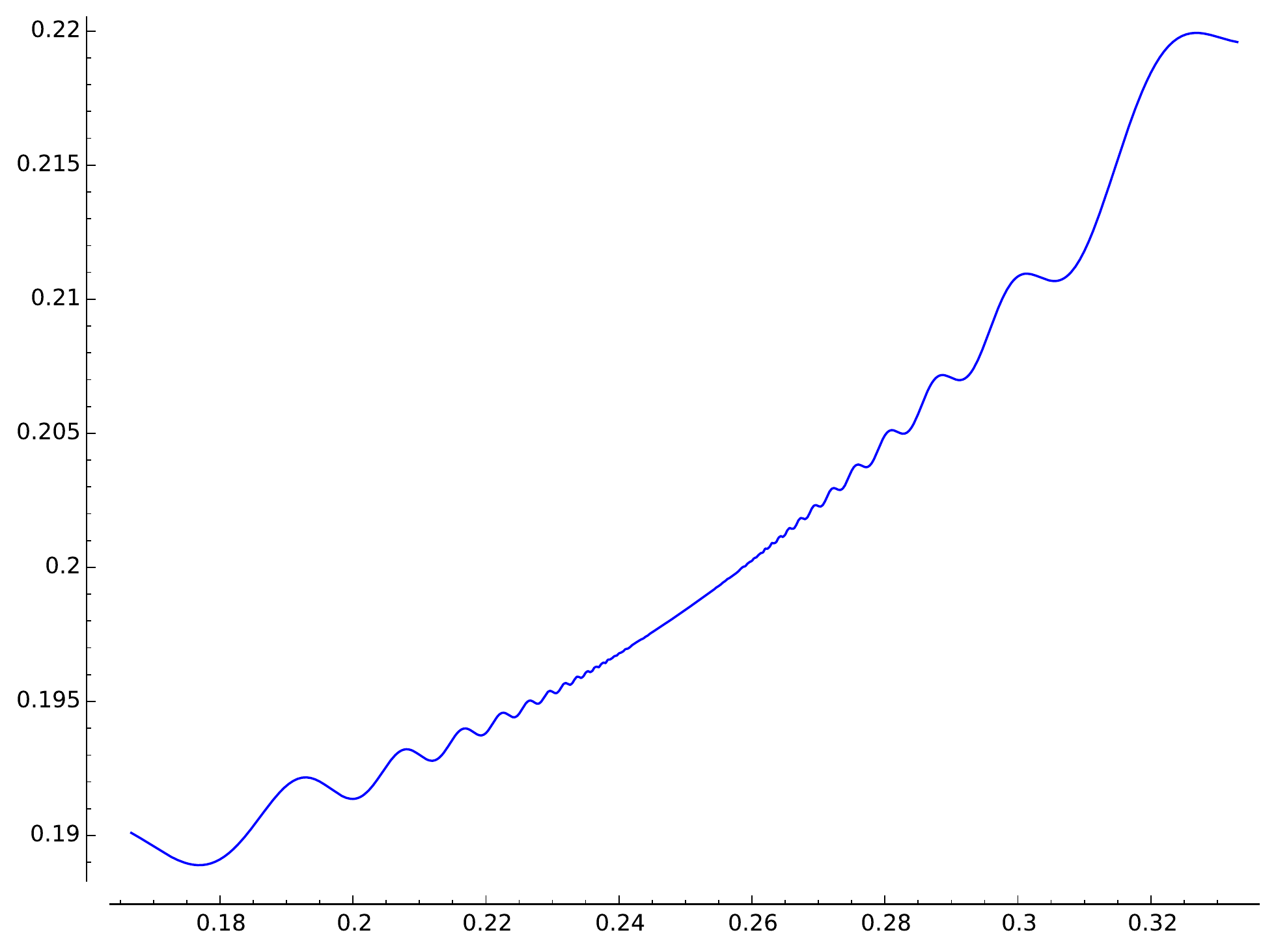}
}
\caption{Graph of $F:\Lambda\to\R$ in Example \ref{ex2_16_07_2015}}\label{figure_ex2_16_07_2015}
\end{figure}

\end{example}

\subsection{A comparison with previous results}\label{ExTanaka}

In this section we compare our results with the main theorems in \cite{Tanaka}, Theorem \ref{Tanaka16} and \ref{Tanaka16_2} here. We start by introducing some definitions, as the setting of  \cite{Tanaka} is more general than our. As a consequence of Theorem \ref{Tanaka16} and \ref{Tanaka16_2} we obtain Corollary \ref{theorem_15_march_2016} that we compare with Corollary \ref{theorem_8_July_2016}, a similar result whose proof follows entirely from Section \ref{PES_prel}.

\begin{definition}[Graph iterated function system]
A GIFS (Graph iterated function system) is defined by a triplet $(G, (J_v), (T_e))$ satisfying the following conditions:
\begin{enumerate}
\item $G=(V,E,i,t)$ is a finite directed multigraph which consists of vertices set $V,$ a directed edges sets $E$ and two functions $i,t:E\to V.$ For each $e\in E,$ $i(e)$ is called the initial vertex of $e$ and $t(e)$ is called the terminal vertex of $e.$ Assume that the graph $G$ is strongly connected and aperiodic.
\item For each $v\in V,$ a subset $J_v\subset \R^D$ is compact and connected so that the interior of $J_v$ is not empty. For every $v,v'\in V$ with $v\neq v'$ we have that $J_{v}$ and $J_{v'}$ are disjoint. 
\item For each $v\in V$ we consider certain connected open sets $O_{v}\subset J_{v},$ so that, for each $e\in E,$ the map $T_e: O_{t(e)}\to O_{i(e)}$ is conformal $\mathcal{C}^{1+\beta}$-diffeomorphism with $\beta>0$ and satisfies $0<\|T'_r(x)\|<1$ for $x\in O_{t(e)},$ and for every $e,e'\in E$ with $e\neq e', i(e)=i(e')$ we have that $T_e J_{t(e)}$ and $T_{e'} J_{t(e')}$ are disjoint.
\end{enumerate}
\end{definition}

\begin{remark}
We stated the definition of GIFS in \cite{Tanaka}. A more general one can be found in \cite{GDMS2003}.
\end{remark}

We notice that IFSs are in particular GIFSs, as they can always be represented by a $1$-vertex GIFS. Moreover, GIFSs may exhibit more general phenomena than IFSs \cite{Boore}.

For GIFSs there is a definition of limit set, similar to the one for IFSs in Lemma \ref{limit_set_definition}.

\begin{definition}[Limit set]
Given a GIFS $(G, (J_v), (T_e))$ we define its limit set by the set $K=\cup_{v\in V} K_v,$ where 
for each $v\in V$ the subset $K_v\subset J_v$ is the unique non-empty compact set such that 
$$
K_v=\cup_{e\in E: i(e)=v} T_e(K_{t(e)}).
$$
\end{definition}

We now introduce a condition on the regularity of the maps $(T_e)$ from \cite{Tanaka}.

\begin{definition}[$(G)_n$ condition]
We say that a family of GIFSs $(G, (J_v), (T_e(\epsilon,\cdot)))$ for $\epsilon>0$ small satisfies the $(G)_n$ condition if  
\begin{enumerate}
\item there exists numbers $\beta>0$ and $\beta(\epsilon)>0$ such that $T_e$ is $\mathcal{C}^{n+1+\beta},$
\item there exists functions $T_{e,1}$ of class  $\mathcal{C}^{n+\beta},$ $\ldots$, $T_{e,n}$ of class $\mathcal{C}^{1+\beta},$ and $\tilde{T}_{e,n}(\epsilon,\cdot)$ of class $\mathcal{C}^{1+\beta(\epsilon)}$ 
defined on $O_{t(e)}$ for each $e\in E$ such that 
$$
T_e(\epsilon,\cdot)=T_e+T_{e,1}\epsilon+\cdots+T_{e,n}\epsilon^n+\tilde{T}_{e,n}(\epsilon,\cdot)\epsilon^n \mbox{ on }J_{t(e)},
$$ 
where $| \tilde{T}_{e,n}(\epsilon,\cdot)|\to 0$ and $\|\frac{\partial}{\partial x} \tilde{T}_{e,n}(\epsilon,\cdot)\|\to 0$ as $\epsilon\to 0.$
\end{enumerate}
\end{definition}

For what follows, let us consider a family of GIFSs $(G, (J_v), (T_e(\epsilon,\cdot)))$ and respective limit sets $K(\epsilon)$ for $\epsilon>0$ small. The main theorem in \cite{Tanaka} is the following.

\begin{theorem}[Theorem 1.1 in  \cite{Tanaka}]\label{Tanaka16}
Assume that the $(G)_n$ condition is satisfied. Then there exist numbers $s_1,\ldots,s_n\in\R$ such that $HD(K(\epsilon))=HD(K)+s_1\epsilon+\cdots+s_n\epsilon^n+o(\epsilon^n)$ in $\R,$
where $HD(K(\epsilon))$ corresponds to the Hausdorff dimension of $K(\epsilon).$
\end{theorem}

In order to state the second main theorem in \cite{Tanaka}, we need to introduce a definition and some notation.

\begin{definition}[$(G)'_n$ condition]
We say that a family of GIFSs $(G, (J_v), (T_e(\epsilon,\cdot)))$ satisfies the $(G)'_n$ condition if it satisfies the $(G)_n$ condition and the small order parts $\tilde{T}_{e,n}(\epsilon,x)$ satisfy
$$
\lim_{\epsilon\to 0}\sup \max_{e\in E} \sup_{x,y\in O_{t(e)}:x\neq y}\frac{\| \frac{\partial}{\partial x} \tilde{T}_{e,n}(\epsilon,x) - \frac{\partial}{\partial x}\tilde{T}_{e,n}(\epsilon,y) \|}{|x-y|^{\beta}}<\infty.
$$
\end{definition}

Let $r\in (0,1)$ be such that $r>\|T'_e\|$ and $r>\|T'_e(\epsilon,\cdot)\|$ for any $e\in E$ and any $\epsilon>0$ small. Denote $E^{\infty}:=\{w=(w_k)_{k=0}^{\infty}\in \prod_{k=0}^{\infty} E:t(w_k)=i(w_{k+1})\mbox{ for all }k\geq 0\}$ and define the shift $\sigma:E^{\infty}\to E^{\infty}.$ Let $\pi: E^{\infty}\to \R^D$ be the projection of the GIFS $(G, (J_v), (T_e))$ defined by $\pi(w):=\cap_{k=0}^{\infty}T_{w_0}\cdots T_{w_k} J_{t(w_k)}$ for $w\in E^{\infty}.$ We define the function $\varphi(w):=\log \| T'_{w_0}(\pi (\sigma w))\|.$ For each $\epsilon>0,$ we denote by $\pi(\epsilon,w)$ the projection of the GIFS $(G, (J_v), (T_e(\epsilon,\cdot)))$ and we denote by $\varphi(\epsilon,w)$ the function $\varphi(\epsilon,w):=\log \| \frac{\partial}{\partial x}T_{w_0}(\epsilon,\pi (\epsilon,\sigma w))\|.$ Finally, we denote by $\mu$ the Gibbs measure of $HD(K)\varphi$ on $E^{\infty}$ and by $\mu(\epsilon,\cdot)$ the Gibbs measure of $HD(K(\epsilon))\varphi(\epsilon,\cdot)$ on $E^{\infty}.$

\begin{theorem}[Theorem 1.2. in  \cite{Tanaka}]\label{Tanaka16_2}
Assume that the $(G)'_n$ condition is satisfied. Choose any $\theta_1\in (r^{\beta},1).$ Then there exists linear functionals $\mu_1,\mu_2,\ldots,\mu_n \in F^*_{\theta_1}(E^{(\infty)},\R),$ and numbers $H_1,H_2,\ldots,H_n\in\R$ such that for each $f\in F_{\theta_1}(E^{(\infty)},\C)$
$$
\begin{aligned}
\mu(\epsilon,f)&=\mu(f)+\mu_1(f)\epsilon+\cdots+\mu_n(f)\epsilon^n+o(\epsilon^n) \mbox{ in }\R\\
h(\mu(\epsilon,\cdot))&=h(\mu)+H_1\epsilon+\cdots+H_n\epsilon^n+o(\epsilon^n) \mbox{ in }\R,
\end{aligned}
$$
where $h(\mu(\epsilon,\cdot))$ denotes the measure-theoretic entropy of the Gibbs measure $\mu(\epsilon,\cdot).$
\end{theorem}

The main ingredients in the proofs of Theorem \ref{Tanaka16} and \ref{Tanaka16_2} are Proposition 2.3 in \cite{Tanaka}, and Theorem 2.1 and Theorem 2.4 in \cite{Tanaka_2011}.\\


In the particular case that the GIFS is also an IFS, we are in conditions to compare our results with Theorem \ref{Tanaka16} and \ref{Tanaka16_2}. We concluded that we can apply our methods to obtain similar results, indeed, we can do the following.\\ 

Consider an IFS $\mathcal T$ as in Definition \ref{03_07_2015_def} such that the sets $T_i^{(\lambda)}[0,1]$ are pairwise  disjoint for $i\in \{ 1 ,\ldots, k\} $ and such that $m=l.$ Using our results in Section \ref{PES_prel}, we can deduce the following result.

\begin{corollary}\label{theorem_15_march_2016}
\begin{enumerate}
\item The dependence 
$\mathcal I\ni \lambda \mapsto HD(\mathcal K(\lambda))$ 
of the Hausdorff dimension of the limit set is $\mathcal{C}^{m-2}.$
\item \label{two_6_jul_2015}For $\alpha\in(0,1)$ small enough so that $2^{\alpha} \|dT_1 \|_{\mathcal C^0}<1$ and $\pi^{(\lambda)}:\mathcal X\to \R$ is $\alpha$-H\"older, the Gibbs measure $\mu_{\varphi}$ of $\varphi=HD(\mathcal K(\lambda)) \psi^{(\lambda)}\in C^{\alpha}(\mathcal X,\R)$ and the measure theoretic entropy $h(\mu_{\varphi})$ of $\mu_{\varphi}$ have both a $\mathcal{C}^{m-2}$ dependence on $\lambda\in\mathcal I,$ when we consider $\mu_{\varphi}$ as an operator on $C^{\alpha}(\mathcal X,\R)^*.$
\end{enumerate}
\end{corollary}

In the same setting, using Theorem \ref{Tanaka16} and \ref{Tanaka16_2} above, instead of our results in Section \ref{PES_prel}, one can deduce a stronger result under slightly different conditions.

\begin{corollary}\label{theorem_8_July_2016}
\begin{enumerate}
\item The dependence 
$\mathcal I\ni \lambda \mapsto HD(\mathcal K(\lambda))$ 
of the Hausdorff dimension of the limit set is $\mathcal{C}^{m-1}.$
\item The Gibbs measure $\mu_{\varphi}$ of $\varphi=HD(\mathcal K(\lambda)) \psi^{(\lambda)}\in C^{\alpha}(\mathcal X,\R)$ and the measure theoretic entropy $h(\mu_{\varphi})$ of $\mu_{\varphi}$ have both a $\mathcal{C}^{m-1}$ dependence on $\lambda\in\mathcal I,$ when we consider $\mu_{\varphi}$ as an operator on $C^{\alpha}(\mathcal X,\R)^*,$ where $\alpha\in (r^\beta,1)$ and $r\in (0,1)$ depends on the rate of contraction of $T^{(\lambda)}.$
\end{enumerate}
\end{corollary}

The difference in the necessary conditions of both corollaries is that in Corollary \ref{theorem_15_march_2016}  the Gibbs measure $\mu_{\varphi}$ is an operator on $C^{\alpha}(\mathcal X,\R)^*,$ where $\alpha\in (r^\beta,1)$ and $r\in (0,1)$ depends on the rate of contraction of $T^{(\lambda)},$ whereas, in Corollary \ref{theorem_8_July_2016}, it is necessary $\alpha\in (0,1)$ small enough so that $2^{\alpha} \|dT_1 \|_{\mathcal C^0}<1$ and $\pi^{(\lambda)}:\mathcal X\to \R$ is $\alpha$-H\"older.


\begin{thebibliography}{0}

\bibitem{Bogachev}
V.I. Bogachev, {\it Measure Theory}  (Vol II. Springer-Verlag, Heidelberg, 2007).

\bibitem{Boore}
G. Boore, {\it Directed graph iterated function systems} (PhD thesis, St Andrews Research Repository, 2011).

\bibitem{Rufus}
R. Bowen, {\it Equilibrium states and the Ergodic theory of Anosov diffeomorphisms}, (Lecture Notes in Math. 470, Springer, Berlin, 1975).

\bibitem{HD_Bowen}
R. Bowen, Hausdorff dimension of quasi-circles, {\it Publ. Math. IHES} {\bf 50(1)}  (1979) 11-25. 

\bibitem{DLLLave}
R. de la Llave and R. Obaya, Regularity of the composition operator in spaces of {H}\"{o}lder functions, {\it Discrete and Continuous Dynamical Systems} {\bf 5(1)} (1999) 157-184.

\bibitem{Falconer}
K. Falconer, {\it Fractal Geometry}, (Wiley, 1990).

\bibitem{furstenber1963}
H. Furstenberg, Non-commuting random products, {\it Trans. Amer. Math. Soc.} {\bf 108} (1963) 377-428.

\bibitem{Geronimo89}
J. Geronimo and D. Hardin, An exact formula for the measure dimension associated with a class of piecewise linear maps, {\it Constr. Approx.} {\bf 5} (1989) 89-98.

\bibitem{Hutchinson}
J. Hutchinson, Fractals and self similarity, {\it Indiana Univ. Math. J.} {\bf 30 } (1981) 713-747.

\bibitem{KantorovichBook}
L.V. Kantorovich and G.P. Akilov, {\it Functional Analysis [in Russian]},  (Nauka, Moscow, 1984).

\bibitem{Kravchenko}
A.S. Kravchenko, Completeness of the space of separable measures in the Kantorovich-Rubinstein 
metric, {\it Siberian Mathematical Journal} {\bf 47(1)} (2006) 68-76.

\bibitem{Ledrappier}
F. Ledrappier, Principe variationnel et systèmes dynamiques symboliques, {\it Zeitschrift für Wahrscheinlichkeitstheorie und Verwandte Gebiete} {\bf 30(3)} (1974) 185-202.

\bibitem{Ledrappier81}
F. Ledrappier, Some relations between dimension and Lyapunov exponents, {\it Comm. Math. Phys.} {\bf 81} (1981) 229-238.

\bibitem{Manning1981}
A. Manning, A relation between Lyapunov exponents, Hausdorff dimension and entropy, {\it Ergodic Theory Dynam. Systems} {\bf 1(4)} (1981) 451-459.

\bibitem{Mane90}
R. Mañé, The Hausdorff dimension of horseshoes of diffeomorphisms of surfaces, {\it Bol. Soc. Bras. Mat.}  {\bf 20(2)} (1990) 1-24.

\bibitem{Manning_HD}
H. McCluskey and A. Manning, Hausdorff dimension for horseshoes, {\it Ergodic Theory and Dynamical Systems}  {\bf 3} (1983) 251-260.

\bibitem{MoranOSC}
P.A.P. Moran, Additive functions of intervals and Hausdorff measure, {\it Proc. Cambridge Philos. Soc.} {\bf 42} (1946) 15-23.

\bibitem{U_Mosco}
U. Mosco, Self-similar measures in quasi-metric spaces, {\it Recent Trends in Nonlinear 
Analysis, Volume 40 of the series Progress in Nonlinear Differential Equations and Their Applications} (2000) 233-248.

\bibitem{GDMS2003}
D. Mauldin and M. Urbanski, {\it Graph Directed Markov Systems}, (Cambridge Univ. Press, 2003).

\bibitem{PP}
W. Parry and M. Pollicott, {\it Zeta functions and the periodic orbit structure of hyperbolic dynamics}, (Asterisque, 1990).

\bibitem{Pollicott2015}
M. Pollicott, Analyticity of dimensions for hyperbolic surface diffeomorphisms, {\it Proc. Amer. Math. Soc.} {\bf 143} (2015) 3465-3474.

\bibitem{Ruelle_82_HD}
D. Ruelle, Repellers for real analytic maps, {\it Ergodic Theory and Dynamical Systems} {\bf 2(1)} (1982) 99-107.

\bibitem{Ruelle_HD_83}
D. Ruelle, Bowen's formula for the Hausdorff dimension of self-similar sets, {\it Progress in Physics} {\bf 7} (1983) 351-358.

\bibitem{ruelle_book_TF}
D. Ruelle, {\it Thermodynamic Formalism: The Mathematical Structure of Equilibrium Statistical Mechanics}, (Addison-Wesley, Cambridge: University Press, 1984).

\bibitem{Tanaka_2011}
H. Tanaka, An asymptotic analysis in thermodynamic formalism, {\it Monatsh. Math} {\bf 164} (2011) 467-486.

\bibitem{Tanaka}
H. Tanaka, Asymptotic perturbation of graph iterated function systems, {\it Journal of Fractal Geometry} (2015).

\bibitem{walters}
P. Walters, A variational principle for the pressure of continuous transformations, {\it Amer. J. Math} {\bf 97} (1975) 937-971.

\bibitem{AFBS65}
E.F. Whittlesey, Analytic functions in Banach spaces, {\it Proceedings of The American Mathematical Society} {\bf 16(5)} (1965).

\bibitem{Young_82_Dim}
L-S. Young, Dimension, entropy and Lyapunov exponents, {\it Ergodic Theory Dynam. Systems} {\bf 2(1)} (1982) 109-124.

\bibitem{MZbook}
M. Zinsmeister, {\it Thermodynamic formalism and holomorphic dynamical systems}, (Volume 2, American Mathematical Society and Société Mathématique de France, 2000).

\end{thebibliography}
\end{document}